\documentclass[reqno]{amsart}

\usepackage{amssymb,latexsym,amsfonts,amsmath}
\usepackage{graphicx}
\usepackage{tikz}
\usetikzlibrary{automata}
\usepackage{dsfont}
\usepackage{diagrams}
\usepackage{comment}

\newtheorem{theorem}{Theorem}[section]
\newtheorem{lemma}[theorem]{Lemma}

\newtheorem{corollary}[theorem]{Corollary}

\newtheorem{definition}[theorem]{Definition}
\newtheorem{example}[theorem]{Example}

\newtheorem{remark}[theorem]{Remark}
\numberwithin{equation}{section}

\newcommand{\params}{{\mathsf{q}}}

\newcommand{\R}{{\mathbb{R}}}

\newcommand{\Ze}{{\mathbb Z}}

\newcommand{\N}{{\mathbb{N}}}

\newcommand{\argmin}{\textrm{arg}\min}

\DeclareMathOperator{\diff}{d}

\newcommand{\sigalg}{\mathcal{F}}
\newcommand{\filtration}{\mathds{F}}

\newcommand{\ol}{\overline}
\newcommand{\Let}{:=}
\newcommand{\EE}{\mathds{E}}
\newcommand{\PP}{\mathds{P}}

\newcommand{\tz}{\times10^}

\begin{document}

\begin{abstract}
The last decade has witnessed significant attention on networked control systems (NCS) due to their ubiquitous presence in industrial applications, and, in the particular case of wireless NCS, because of their architectural flexibility and low installation and maintenance costs. In wireless NCS the communication between sensors, controllers, and actuators is supported by a communication channel that is likely to introduce variable communication delays, packet losses, limited bandwidth, and other practical non-idealities leading to numerous technical challenges. 
Although stability properties of NCS have been investigated extensively in the literature, results for NCS under more complex and general objectives, 
and in particular results dealing with verification or controller synthesis for logical specifications, 
are much more limited.     
This work investigates how to address such complex objectives by constructively deriving symbolic models of NCS, 
while encompassing the mentioned network non-idealities. 
The obtained abstracted (symbolic) models can then be employed to synthesize hybrid controllers enforcing rich logical specifications over the concrete NCS models. 
Examples of such general specifications include properties expressed as formulae in linear temporal logic (LTL) or as automata on infinite strings. 
We thus provide a general synthesis framework that can be flexibly adapted to a number of NCS setups. We illustrate the effectiveness of the results over some case studies.
\end{abstract}

\title[Symbolic Abstractions of Networked Control Systems]{Symbolic Abstractions of Networked Control Systems}

\author[M. Zamani]{Majid Zamani$^1$} 
\author[M. Mazo Jr]{Manuel Mazo Jr$^2$}
\author[M. Khaled]{Mahmoud Khaled$^1$} 
\author[A. Abate]{Alessandro Abate$^4$} 
\address{$^{1,3}$Department of Electrical and Computer Engineering, Technical University of Munich, D-80290 Munich, Germany.}
\email{\{zamani,khaled.mahmoud\}@tum.de}
\urladdr{http://www.hcs.ei.tum.de}
\address{$^2$Delft Center for Systems and Control, Delft University of Technology, Mekelweg 2, 2628 CD, Delft, The Netherlands.}
\email{m.mazo@tudelft.nl}
\urladdr{http://www.dcsc.tudelft.nl/$\sim$mmazo}
\address{$^4$Department of Computer Science, University of Oxford, Wolfson Building, Parks Road, Oxford, OX1 3QD, UK.}
\email{alessandro.abate@cs.ox.ac.uk}
\urladdr{https://www.cs.ox.ac.uk/people/alessandro.abate}

\maketitle

\section{Introduction}
Over the last decade the analysis and synthesis of networked control systems (NCS) have received significant attention.  
NCS are ubiquitous in most industrial applications due to their many advantages over traditional control systems,   
such as increased architectural flexibility and reduced installation and maintenance costs, particularly for wireless NCS.  
The numerous non-idealities of the network in an NCS introduce new challenges for the analysis of the behavior (such as the stability) of the plant, and for the synthesis of new control schemes. The various non-idealities of the network can be broadly categorized as follows: 
(i) quantization errors; 
(ii) packet dropouts; 
(iii) time-varying sampling/transmission intervals; 
(iv) time-varying communication delays; 
and (v) communication constraints (e.g. scheduling protocols). 
The limited bandwidth of the network does not require a separate classification as it is captured by a combination of quantization errors (i) and the communication delays (iv). 
As pointed out later in the paper, category (ii) can also be incorporated in category (iv), as long as the maximum number of subsequent dropouts over the network is bounded \cite{heemels}. 

Recently, there have been many studies focused mostly on the stability properties of NCS: in \cite{bauer} (iii)-(v) are simultaneously considered; in \cite{gao} (i), (ii), and (iv) are taken into account; \cite{alur2} studies (ii) and (v); \cite{antunes} focuses on (ii) and (iii); in \cite{marieke,nathan} (ii)-(iv) are considered; and finally in \cite{nesic} (i), (iii), and (v) are taken into account. Despite all the progress on the stability analysis of NCS as reported in \cite{bauer,gao,alur2,antunes,marieke,nathan,nesic}, there are no mature results in the literature dealing with more complex objectives, such as model verification or formal (controller) synthesis for richer properties expressed as temporal logic specifications \cite{katoen08}. Examples of those specifications include linear temporal logic (LTL) formulae or automata over infinite strings \cite{katoen08}, 
which cannot be investigated with existing approaches for NCS.  
A promising direction to study these complex properties is the use of \emph{symbolic models} \cite{paulo}. A symbolic model is an abstract description of the original (concrete) dynamical model, where each abstract state (or symbol) corresponds to an aggregate of continuous states in the concrete model. 
When a finite symbolic model is obtained and is formally related to the original model via the notions of (alternating) approximate (bi)simulations \cite{paulo} or feedback refinement relations \cite{Matthias_Gunther}, one can leverage algorithmic machinery for controller synthesis of symbolic systems \cite{MalerPnueliSifakis95} to automatically synthesize hybrid controllers for the original, concrete model \cite{paulo}. 

To the best of our knowledge, the first and only results in the literature on the construction of symbolic models for NCS are \cite{pola4,pola5}: 
these results provide symbolic models for NCS obtained via gridding techniques (discretization of state and control sets); 
they simultaneously consider the network non-idealities (i), (ii), and (iv); 
they address symbolic control design with objectives only expressed in terms of non-deterministic
automata;
the possibility of out-of-order packet arrivals is not considered; 
they exclusively consider static (i.e. memoryless) symbolic controllers; 
and, furthermore, 
in order to apply standard algorithms for verification and synthesis to the obtained symbolic model
often the given specification requires an additional reformulation over an extended state-space, 
which can lead to significant computation overheads.

In this article we provide a general construction of symbolic models for NCS, 
which can directly employ available and well investigated symbolic models from the literature that are obtained exclusively for the plant (that is, without the need to encompass the presence of the network explicitly in the construction).  
As such, one can directly leverage existing results to obtain symbolic models for the plant, such as grid-based approaches in \cite{girard2,majid,Matthias_Gunther}, 
recent results in \cite{majid10,ZTA1} that do not require state-space discretization but only input-set discretization, 
or formula-guided (non-grid-based) approaches in \cite{boyan}. 
In this work we show that, 
having a symbolic model of the plant, 
one can then construct symbolic models for the overall NCS. 
As a consequence, 
as long as there exists some type of symbolic abstraction of the plant, 
one can always use the results provided in this article to construct symbolic models for the overall, complex NCS. 
As a relevant side result, 
the techniques discussed in this paper can also be used for models of stochastic plants,   
in view of recent literature providing symbolic models for such systems \cite{majid10,ZTA1,majid7,majid8}. 
In this work, we explicitly consider the network non-idealities (i), (ii), and (iv) acting on the NCS simultaneously. 
We further consider possible out-of-order packet arrivals and message rejections, 
i.e. the effect of older data being neglected because more recent one is available.
Let us also remark that this work is not limited to problems where the controller is static. 
As a result, 
without requiring any specific reformulation, we enable the study of large classes of logical specifications, 
such as those expressed as general LTL formulae or as automata on infinite strings, 
which are often shown to require dynamic (i.e. with memory) symbolic controllers (cf. the example section) \cite{katoen08}.


This paper presents a detailed and mature description of the results announced in \cite{majid13}, 
including a detailed discussion on dealing with the quantized measurements, on the symbolic controller synthesis and refinement, and on the space complexity, and several case studies. Furthermore, we have added a section on related work and provided a detailed comparison with the results in \cite{pola4,pola5}.

\section{Notations and Basic Concepts} 
\subsection{Notations}\label{II.A}
The identity map on a set $A$ is denoted by $1_{A}$. 
The symbols $\N$, $\N_0$, $\Ze$, $\R$, $\R^+$, and $\R_0^+$ denote the set of natural, nonnegative integer, integer, real, positive, and nonnegative real numbers, respectively. Given a set $A$, define $A^{n+1}=A\times{A}^n$ for any $n\in\N$. 
Given a vector \mbox{$x\in\mathbb{R}^{n}$}, we denote by $x_{i}$ the $i$-th element of $x$, and by $\Vert x\Vert$ the infinity norm of $x$.
Given an interval $[a,b]\subseteq\R$ with $a\leq{b}$, we denote by $[a;b]$ the set $[a,b]\cap\N$. We denote by $[\R^n]_{\eta}=\left\{a\in \R^n\mid a_{i}=k_{i}\eta,~k_{i}\in\mathbb{Z},~i=1,\ldots,n\right\}$.

Given a measurable function \mbox{$f:\mathbb{R}_{0}^{+}\rightarrow\mathbb{R}^n$}, the (essential) supremum of $f$ is denoted by $\Vert f\Vert_{\infty}$, where $\Vert f\Vert_{\infty}:=\text{(ess)sup}\{\Vert f(t)\Vert,t\geq0\}$. 
A continuous function \mbox{$\gamma:\mathbb{R}_{0}^{+}\rightarrow\mathbb{R}_{0}^{+}$} is said to belong to class $\mathcal{K}$ if it is strictly increasing and \mbox{$\gamma(0)=0$}; $\gamma$ is said to belong to class $\mathcal{K}_{\infty}$ if \mbox{$\gamma\in\mathcal{K}$} and $\gamma(r)\rightarrow\infty$ as $r\rightarrow\infty$. A continuous function \mbox{$\beta:\mathbb{R}_{0}^{+}\times\mathbb{R}_{0}^{+}\rightarrow\mathbb{R}_{0}^{+}$} is said to belong to class $\mathcal{KL}$ if, for each fixed $s$, the map $\beta(r,s)$ belongs to class $\mathcal{K}$ with respect to $r$ and, for each fixed nonzero $r$, the map $\beta(r,s)$ is decreasing with respect to $s$ and $\beta(r,s)\rightarrow 0$ as \mbox{$s\rightarrow\infty$}. We identify a relation \mbox{$R\subseteq A\times B$} with the map \mbox{$R:A \rightarrow 2^{B}$} defined by $b\in R(a)$ iff \mbox{$(a,b)\in R$}. Given a relation \mbox{$R\subseteq A\times B$}, $R^{-1}$ denotes the inverse relation defined by \mbox{$R^{-1}=\{(b,a)\in B\times A:(a,b)\in R\}$}. When $R$ is an equivalence relation\footnote{An equivalence relation $R\subseteq X\times{X}$ is a binary relation on a set $X$ if it is reflexive, symmetric, and transitive.} on a set $A$, we denote by $[a]$ the equivalence class corresponding to the element $a\in{A}$, by $A/R$ the set of all equivalence classes (quotient set), and by $\pi_R:A\rightarrow A/R$ the natural projection map taking a point $a\in{A}$ to its equivalence class $\pi(a)=[a]\in A/R$. 

\subsection{Control systems}\label{control_system}
The class of control systems that we consider in this paper is formalized in
the following definition.
\begin{definition}
\label{Def_control_sys}A \textit{control system} $\Sigma$ is a tuple
$\Sigma=(\mathbb{R}^{n},\mathsf{U},\mathcal{U},f)$,
where:
\begin{itemize}
\item $\mathbb{R}^{n}$ is the state space;
\item $\mathsf{U}\subseteq\mathbb{R}^{m}$ is the bounded input set;
\item $\mathcal{U}$ is a subset of the set of all measurable functions of time, from intervals of the form \mbox{$]a,b[\subseteq\mathbb{R}$} to $\mathsf{U}$, with $a<0$ and $b>0$; 
\item \mbox{$f:\mathbb{R}^{n}\times \mathsf{U}\rightarrow\mathbb{R}^{n}$} is a continuous map
satisfying the following Lipschitz assumption: for every compact set
\mbox{$Q\subset\mathbb{R}^{n}$}, there exists a constant $Z\in\mathbb{R}^+$ such that $\Vert
f(x,u)-f(y,u)\Vert\leq Z\Vert x-y\Vert$ for all $x,y\in Q$ and all $u\in \mathsf{U}$.
\end{itemize}
\end{definition}

A locally absolutely continuous curve \mbox{$\xi:]a,b[\rightarrow\mathbb{R}^{n}$} is said to be a
\textit{trajectory} of $\Sigma$ if there exists $\upsilon\in\mathcal{U}$
satisfying:
\begin{equation}\nonumber
\dot{\xi}(t)=f\left(\xi(t),\upsilon(t)\right),
\end{equation}
for almost all $t\in$ $]a,b[$. Although we have defined trajectories over open domains, we shall as well refer to trajectories \mbox{${\xi:}[0,t]\rightarrow\mathbb{R}^{n}$} defined on closed domains $[0,t]$, $t\in\mathbb{R}^{+}$, with the understanding of the existence of a trajectory \mbox{${\xi}^{\prime}:]a,b[\rightarrow\mathbb{R}^{n}$} such that \mbox{${\xi}={\xi}^{\prime}|_{[0,t]}$} with $a<0$ and $b>t$. We also write $\xi_{x\upsilon}(t)$ to denote the point reached at time $t$
under the input $\upsilon$ from the initial condition $x=\xi_{x\upsilon}(0)$; the point $\xi_{x\upsilon}(t)$ is
uniquely determined due to the assumptions on $f$ \cite{sontag1}. A control system $\Sigma$ is said to be forward complete if every trajectory is defined on an interval of the form $]a,\infty[$ \cite{sontag}.

\subsection{Notions of stability and of completeness}
Some of the existing results recalled in this paper require certain stability properties (or lack thereof) on $\Sigma$. 
First, we recall a stability property, introduced in \cite{angeli}, as defined next. 
\begin{definition}
\label{dGAS}
A control system $\Sigma$ is incrementally input-to-state stable ($\delta$-ISS) if it is forward complete and there exists a $\mathcal{KL}$ function $\beta$ and a $\mathcal{K}_\infty$ function $\gamma$ such that for any $t\in{\mathbb{R}_0^+}$, any $x,\hat x\in\R^n$, and any $\upsilon,\hat \upsilon\in\mathcal{U}$, the following condition is satisfied:  
\begin{equation}
\left\Vert \xi_{x\upsilon}(t)-\xi_{\hat x{\hat \upsilon}}(t)\right\Vert\leq\beta\left(\left\Vert x-\hat x \right\Vert,t\right)+\gamma\left(\Vert\upsilon-\hat\upsilon\Vert_\infty\right). \label{delta_PGAS}
\end{equation}
\end{definition}

Next we recall a completeness property, introduced in \cite{majid}, which can be satisfied by larger classes of (even unstable) control systems.
\begin{definition}
\label{dFC}
A control system $\Sigma$ is incrementally forward complete ($\delta$-FC) if it is forward complete and there exist continuous functions $\beta:\R_0^+\times\R_0^+\rightarrow\R_0^+$ and $\gamma:\R_0^+\times\R_0^+\rightarrow\R_0^+$ such that for each fixed $s$, the functions $\beta(r,s)$ and $\gamma(r,s)$ belong to class $\mathcal{K}_{\infty}$ with respect to $r$, and for any $t\in{\mathbb{R}_0^+}$, any $x,\hat x\in\R^n$, and any $\upsilon,\hat\upsilon\in\mathcal{U}$, the following condition is satisfied:  
\begin{equation}
\left\Vert \xi_{x\upsilon}(t)-\xi_{\hat x{\hat\upsilon}}(t)\right\Vert\leq\beta\left(\left\Vert x-\hat x \right\Vert,t\right)+\gamma\left(\Vert\upsilon-\hat\upsilon\Vert_\infty,t\right). \label{delta_FC}
\end{equation}
\end{definition}

As explained in \cite[Remark 2.3]{majid}, $\delta$-FC implies uniform continuity
of the map $\phi_t:\R^n\times \mathcal{U}\rightarrow \R^n$ defined by $\phi_t(x,\upsilon)=\xi_{x\upsilon}(t)$ for any fixed $t\in\R_0^+$.


We refer the interested readers to the results in \cite{angeli} (resp. \cite{majid}) providing a characterization (resp. description) of $\delta$-ISS (resp. $\delta$-FC) in terms of the existence of so-called \emph{incremental Lyapunov functions}.

\section{Systems \& Approximate Equivalence Notions}\label{symbolic_model}
We now recall the notion of {\it system}, 
as introduced in \cite{paulo}, 
that we later use to describe NCS as well as their symbolic abstractions. 

\begin{definition}\label{system}
A system $S$ is a tuple $S=\left(X,X_0,U,\rTo,Y,H\right)$ consisting of:
a (possibly infinite) set of states $X$; a (possibly infinite) set of initial states $X_0\subseteq{X}$;
a (possibly infinite) set of inputs $U$;
a transition relation $\rTo\subseteq X\times U\times X$;
a set of outputs $Y$;
and
an output map $H:X\rightarrow Y$.
\end{definition}

A transition $(x,u,x')\in\rTo$ is also denoted by $x\rTo^{u} x'$. 
If $x\rTo^{u} x'$, state $x'$ is called a $u$-successor of state $x$. We denote by $\textbf{Post}_u(x)$ the set of all \mbox{$u$-successors} of a state $x$, and by $U(x)$ the set of inputs $u\in{U}$ for which $\textbf{Post}_u(x)$ is nonempty. We denote by $\mathcal{T}(U,Y)$ the set of all systems associated to a set of
inputs $U$ and a set of outputs $Y$.
A system $S$ is said to be:
\begin{itemize}
\item \textit{metric}, if the output set $Y$ is equipped with a metric
$\mathsf{d}:Y\times Y\rightarrow\mathbb{R}_{0}^{+}$;
\item \textit{finite} (or \textit{symbolic}), if $X$ and $U$ are finite sets;
\item \textit{countable}, if $X$ and $U$ are countable sets;
\item \textit{deterministic}, if for any state $x\in{X}$ and any input $u\in{U(x)}$, $\left\vert\textbf{Post}_u(x)\right\vert=1$; 
\item \textit{nondeterministic}, if there exist a state $x\in{X}$ and an input $u\in{U}$ such that $\left\vert\textbf{Post}_u(x)\right\vert>1$;
\end{itemize}

Given a system $S=(X,X_0,U,\rTo,Y,H)$, we denote by $\left\vert{S}\right\vert$ the size of $S$, defined as $\left\vert{S}\right\vert:=\left\vert\rTo\right\vert$, which is equal to the total number of transitions in $S$. Note that it is more reasonable to consider $\left\vert\rTo\right\vert$ as the size of $S$ rather than $\left\vert X\right\vert$ because in practice it is the transitions of $S$ that are required to be stored rather than just the states of $S$.


We recall the notions of (alternating) approximate (bi)simulation relation, introduced in \cite{girard,pola1}, 
which are useful to relate properties of NCS to those of their symbolic models. 
First we recall the notion of approximate (bi)simulation relation, introduced in \cite{girard}.

\begin{definition}\label{ASR}
Let \mbox{$S_{a}=(X_{a},X_{a0},U_{a},\rTo_{a},Y_a,H_{a})$} and 
\mbox{$S_{b}=(X_{b},X_{b0},U_{b},\rTo_{b},Y_b,H_{b})$} be metric systems with the 
same output sets $Y_a=Y_b$ and metric $\mathsf{d}$.
For $\varepsilon\in\mathbb{R}_0^{+}$, 
a relation 
\mbox{$R\subseteq X_{a}\times X_{b}$} is said to be an $\varepsilon$-approximate simulation relation from $S_{a}$ to $S_{b}$ 
if the following three conditions are satisfied:
\begin{itemize}
\item[(i)] for every $x_{a0}\in{X_{a0}}$, there exists $x_{b0}\in{X_{b0}}$ with $(x_{a0},x_{b0})\in{R}$;
\item[(ii)] for every $(x_{a},x_{b})\in R$, we have \mbox{$\mathsf{d}(H_{a}(x_{a}),H_{b}(x_{b}))\leq\varepsilon$};
\item[(iii)] for every $(x_a,x_b)\in R$, the existence of $x_a\rTo_{a}^{u_a}x'_a$ in $S_a$ implies the existence of $x_b\rTo_{b}^{u_b}x'_b$ in $S_b$ satisfying $(x'_a,x'_b)\in R$.
\end{itemize}  
A relation $R\subseteq X_a\times X_b$ is said to be an $\varepsilon$-approximate bisimulation relation between $S_a$ and $S_b$
if $R$ is an $\varepsilon$-approximate simulation relation from $S_a$ to $S_b$ and
$R^{-1}$ is an $\varepsilon$-approximate simulation relation from $S_b$ to $S_a$.\\ 
System $S_{a}$ is $\varepsilon$-approximately simulated by $S_{b}$, 
denoted by \mbox{$S_{a}\preceq_{\mathcal{S}}^{\varepsilon}S_{b}$}, if there exists
an $\varepsilon$-approximate simulation relation from $S_{a}$ to $S_{b}$.
System $S_{a}$ is $\varepsilon$-approximately bisimilar to $S_{b}$, denoted by \mbox{$S_{a}\cong_{\mathcal{S}}^{\varepsilon}S_{b}$}, if there exists
an $\varepsilon$-approximate bisimulation relation between $S_{a}$ and $S_{b}$.
\end{definition}


As explained in \cite{pola1}, for nondeterministic systems we need to consider relationships that explicitly capture the adversarial nature
of nondeterminism. Furthermore, these types of relations become crucial to enable the refinement of symbolic controllers~\cite{paulo}.

\begin{definition}\label{AASR}
Let \mbox{$S_{a}=(X_{a},X_{a0},U_{a},\rTo_{a},Y_a,H_{a})$} and 
\mbox{$S_{b}=(X_{b},X_{b0},U_{b},\rTo_{b},Y_b,H_{b})$} be metric systems with the 
same output sets $Y_a=Y_b$ and metric $\mathsf{d}$.
For $\varepsilon\in\mathbb{R}_0^{+}$, 
a relation 
\mbox{$R\subseteq X_{a}\times X_{b}$} is said to be an alternating $\varepsilon$-approximate simulation relation from $S_{a}$ to $S_{b}$ 
if conditions (i) and (ii) in Definition \ref{ASR}, as well as the following condition, are satisfied:
\begin{itemize}
\item[(iii)] for every $(x_a,x_b)\in R$ and for every $u_a\in U_a\left(x_a\right)$ there exists some $u_b\in U_b\left(x_b\right)$ such that for every $x'_b\in\textbf{Post}_{u_b}(x_b)$ there exists $x'_a\in\textbf{Post}_{u_a}(x_a)$ satisfying $(x'_a,x'_b)\in R$.
\end{itemize}  
A relation $R\subseteq X_a\times X_b$ is said to be an alternating $\varepsilon$-approximate bisimulation relation between $S_a$ and $S_b$
if $R$ is an alternating $\varepsilon$-approximate simulation relation from $S_a$ to $S_b$ and
$R^{-1}$ is an alternating $\varepsilon$-approximate simulation relation from $S_b$ to $S_a$. \\
System $S_{a}$ is alternatingly $\varepsilon$-approximately simulated by $S_{b}$, 
denoted by \mbox{$S_{a}\preceq_{\mathcal{AS}}^{\varepsilon}S_{b}$}, if there exists
an alternating $\varepsilon$-approximate simulation relation from $S_{a}$ to $S_{b}$.
System $S_{a}$ is alternatingly $\varepsilon$-approximately bisimilar to $S_{b}$, denoted by \mbox{$S_{a}\cong_{\mathcal{AS}}^{\varepsilon}S_{b}$}, if there exists
an alternating $\varepsilon$-approximate bisimulation relation between $S_{a}$ and $S_{b}$.
\end{definition}


It can be readily seen that the notions of approximate (bi)simulation relation and of alternating approximate (bi)simulation relation coincide when the systems involved are deterministic, in the sense of Definition \ref{system}.

Let us introduce a metric system $S_\tau(\Sigma)\Let(X_\tau,X_{\tau0},U_\tau,\rTo_{\tau},Y_\tau,H_{\tau})$, 
which captures all the information contained in the forward complete control system $\Sigma$ at sampling times $k\tau$, $\forall k\in\N_0$: $X_\tau=\R^n$, $X_{\tau0}=\R^n$, $U_\tau=\mathcal{U}$, $Y_\tau=\R^n/Q$ for some given equivalence relation $Q\subseteq X_\tau\times X_\tau$, $H_\tau=\pi_Q$, and

\begin{itemize}
\item $x_\tau\rTo_{\tau}^{\upsilon_\tau}x'_\tau$ if there exists a trajectory $\xi_{x_\tau\upsilon_\tau}:[0,\tau]\rightarrow\R^n$ of $\Sigma$ satisfying $\xi_{x_\tau\upsilon_\tau}(\tau)=x'_\tau$.
\end{itemize}
Notice that the set of states and inputs of $S_\tau(\Sigma)$ are uncountable and that $S_\tau(\Sigma)$ is a deterministic system in the sense of Definition \ref{system} since (cf. Subsection \ref{control_system}) the trajectory of $\Sigma$ is uniquely determined. We also assume that the output set $Y_\tau$ is equipped with a metric $\mathsf{d}_{Y_\tau}:Y_\tau\times Y_\tau\rightarrow\R_0^+$.

We refer the interested readers to \cite{girard2,majid,majid10,ZTA1} proposing results on the existence of symbolic abstractions $S_\params(\Sigma)\Let(X_\params,X_{\params0},U_\params,\rTo_\params,Y_\params,H_\params)$ for $S_\tau(\Sigma)$. In particular, the results in \cite{girard2,majid10,ZTA1} and \cite{majid} provide symbolic abstractions $S_\params(\Sigma)$ for $\delta$-ISS and $\delta$-FC control systems $\Sigma$, respectively, such that $S_\params(\Sigma)\cong_{\mathcal{S}}^{\varepsilon}S_\tau(\Sigma)$ (equivalently $S_\params(\Sigma)\cong_{\mathcal{AS}}^{\varepsilon}S_\tau(\Sigma)$)\footnote{\label{coincide}Recall that the notions of alternating approximate (bi)simulation and approximate (bi)simulation relation coincide when the systems involved are deterministic.} and $S_\params(\Sigma)\preceq_{\mathcal{AS}}^{\varepsilon}S_\tau(\Sigma)\preceq_{\mathcal{S}}^{\varepsilon}S_\params(\Sigma)$, respectively. The results in \cite{girard2,majid} assume that $Q$ is the identity relation in the definition of $S_\tau(\Sigma)$,  implying that $Y_\tau=\R^n$ and $\pi_Q=1_{\R^n}$, $\mathcal{U}$ is the set of piecewise constant curves over intervals of length $\tau$ (cf. equation \eqref{input_set}), and the metric $\mathsf{d}_{Y_\tau}$ is the natural infinity norm metric. While the abstraction results in \cite{girard2,majid} are based on state-space discretization, the ones in \cite{majid10,ZTA1} do not require any state-space discretization, and are potentially more efficient than those in \cite{girard2,majid} when dealing with high-dimensional 
plants.

\begin{remark}
Consider a metric system $S_\tau(\Sigma)$ admitting an abstraction $S_\params(\Sigma)$. Since the plant $\Sigma$ is forward complete, one can readily verify that given any state $x_\tau\in X_\tau$, there always exists a $\upsilon_\tau$-successor of $x_\tau$, for any $\upsilon_\tau\in U_\tau$. Hence, $U_\tau(x_\tau)=U_\tau$ for any $x_\tau\in X_\tau$. Therefore, without loss of generality, one can also assume that $U_\params(x_\params)=U_\params$ for any $x_\params\in X_\params$.

\end{remark}

\section{Models of Networked Control Systems}\label{NCS}
\begin{figure}
\centering
\includegraphics[width=8.5cm]{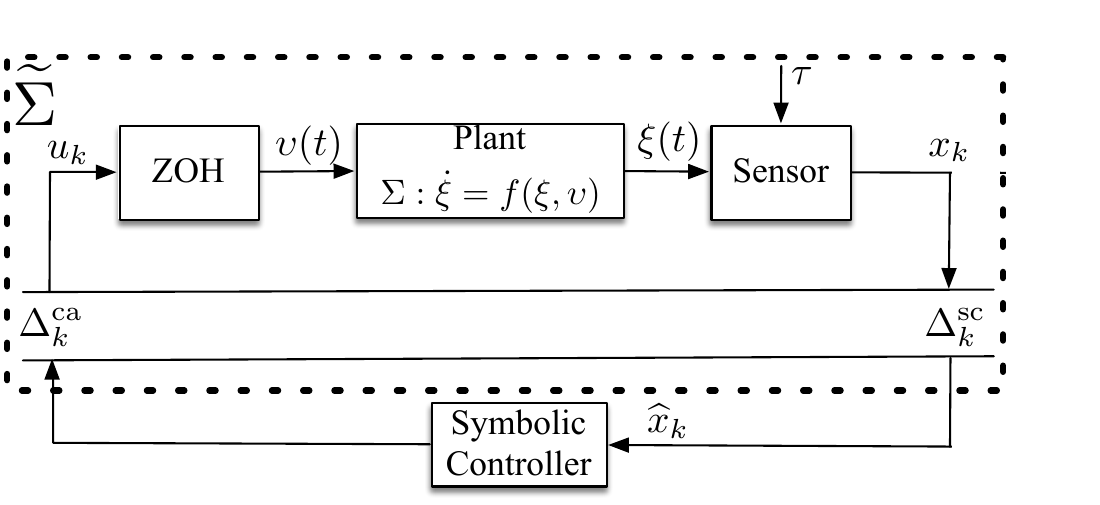}
\caption {Schematics of a networked control system $\widetilde\Sigma$.}\label{fig2}
\end{figure}

Consider an NCS $\widetilde\Sigma$ as depicted schematically in Figure \ref{fig2}, 
and similar to those discussed in \cite[Figure 1]{marieke}, \cite[Figure 1]{nathan}, and \cite[Figure 1]{pola4}. 
The NCS $\widetilde\Sigma$ includes a plant $\Sigma$, a time-driven sampler, and an event-driven zero-order-hold (ZOH), all of which are described in more detail later. 
The NCS consists of a forward complete plant $\Sigma=(\mathbb{R}^{n},\mathsf{U},\mathcal{U},f)$, 
which is connected to a symbolic controller, explained in more detail in the next subsection, over a communication network that induces delays ($\Delta^{\text{sc}}$ and $\Delta^{\text{ca}}$). The state measurements of the plant are sampled by a time-driven sampler at times $s_k\Let k\tau$, $k\in\N_0$, and we denote $x_k\Let \xi(s_k)$. 
The discrete-time control values computed by the symbolic controller at times $s_k$ are denoted by $u_k$. Time-varying network-induced delays, i.e. the sensor-to-controller delay ($\Delta_k^{\text{sc}}$) and the controller-to-actuator delay ($\Delta_k^{\text{ca}}$), are included in the model. Moreover, packet dropouts in both channels of the network can be incorporated in the delays $\Delta_k^{\text{sc}}$ and $\Delta_k^{\text{ca}}$ (increasing them), as long as the maximum number of subsequent dropouts over the network is bounded \cite{heemels}; we refer the interested readers to \cite{heemels} for more detailed information. Finally, the time-varying computation time needed to evaluate the symbolic controller is incorporated into $\Delta_k^{\text{ca}}$ \cite{heemels}. 
We assume that the time-varying delays are bounded and are integer multiples of the sampling time $\tau$, i.e. $\Delta_k^{\text{sc}}\Let N_k^{\text{sc}}\tau$, where $N_k^{\text{sc}}\in\left[N^{\text{sc}}_{\min};N^{\text{sc}}_{\max}\right]$, and $\Delta_k^{\text{ca}}\Let N_k^{\text{ca}}\tau$, where $N_k^{\text{ca}}\in\left[N^{\text{ca}}_{\min};N^{\text{ca}}_{\max}\right]$, for some $N^{\text{sc}}_{\min},N^{\text{sc}}_{\max},N^{\text{ca}}_{\min},N^{\text{ca}}_{\max}\in\N_0$. Note that this assumption implies perfect clock synchronization in the network. Nonetheless, with current technologies it is possible to reach synchronization at the micro-second level (even on wireless networks), see e.g.~\cite{dai2004tsync,elson2002fine}. Thus, one can assume that synchronization errors in general have a rather small effect that 
could be easily incorporated in the form of bounded sensor noise (due to signals excursion in that time interval). 
%
Furthermore, we model the occurrence of message rejection, i.e. the effect of older data being neglected because more recent data is available before the older data arrival, as done in \cite{marieke,nathan}. The zero-order-hold (ZOH) function (see Figure \ref{fig2}) is placed before the plant $\Sigma$ to transform the discrete-time control inputs $u_k$, $k\in\N_0$, to a continuous-time control input $\upsilon(t)=u_{k^*(t)}$, where $k^*(t)\Let\max\left\{k\in\N_0\,\,|\,\,s_k+\Delta^{\text{ca}}_k\leq{t}\right\}$. As argued in \cite{marieke,nathan}, within the sampling interval $\left[s_k,s_{k+1}\right[$, $\upsilon(t)$ can be explicitly described by
\begin{align}\label{input_com}
\upsilon(t)=u_{k+j^k_*-N^{\text{ca}}_{\max}},~~\text{for}~t\in\left[s_k,s_{k+1}\right[,
\end{align}
where $j^k_*\in\left[0;N^{\text{ca}}_{\max}-N^{\text{ca}}_{\min}\right]$, the required time-indexing shift needed to determine the control input available at the ZOH, is defined as:
\begin{align}\label{switching_time}
j^k_*=\lambda\left(\widehat{N}_{N^{\text{ca}}_{\min}},\widehat{N}_{N^{\text{ca}}_{\min}+1},\ldots,\widehat{N}_{N^{\text{ca}}_{\max}}\right),
\end{align}
and where $\widehat{N}_\ell$, for $\ell\in\left[N^{\text{ca}}_{\min};N^{\text{ca}}_{\max}\right]$, is the delay suffered by the control packet sent $\ell$ samples beforehand, namely $\widehat{N}_{N^{\text{ca}}_{\max}-i}=N^{\text{ca}}_{k-N^{\text{ca}}_{\max}+i}$ for any $i\in\left[0;N^{\text{ca}}_{\max}-N^{\text{ca}}_{\min}\right]$, and
\begin{align}\nonumber
\lambda(\widehat{N}_{N^{\text{ca}}_{\min}},\ldots,\widehat{N}_{N^{\text{ca}}_{\max}}):=\max\{\argmin_{j}\kappa(j,\widehat{N}_{N^{\text{ca}}_{\min}},\ldots,\widehat{N}_{N^{\text{ca}}_{\max}})\},
\end{align}
where
\begin{align}\nonumber
\kappa(j,\widehat{N}_{N^{\text{ca}}_{\min}},\ldots,\widehat{N}_{N^{\text{ca}}_{\max}}):=
\min\Big\{&\max\{0,\widehat{N}_{N^{\text{ca}}_{\max}-j}+j-N^{\text{ca}}_{\max}\},\max\{0,\widehat{N}_{N^{\text{ca}}_{\max}-1-j}+j-N^{\text{ca}}_{\max}+1\},\ldots,\\\notag&\max\{0,\widehat{N}_{N^{\text{ca}}_{\min}}-N^{\text{ca}}_{\min}\},1\Big\},
\end{align}
with $j\in\left[0;N^{\text{ca}}_{\max}-N^{\text{ca}}_{\min}\right]$. Note that the expression for the continuous-time control input in \eqref{input_com} and \eqref{switching_time} takes into account the possible out-of-order packet arrivals and message rejection. For example, in Figure \ref{reject}, the time-delays in the controller-to-actuator branch of the network are allowed to take values in $\{\tau,2\tau,3\tau\}$, resulting in a message rejection at time $s_{k+2}$. We refer the interested readers to \cite[Lemma 1]{marieke} to understand how the proposed choices for $j^k_*$ \eqref{switching_time}, $\lambda$, and $\kappa$, can take care of the possible out-of-order packet arrivals and message rejections.

\begin{figure}
\centering
\includegraphics[width=8.5cm]{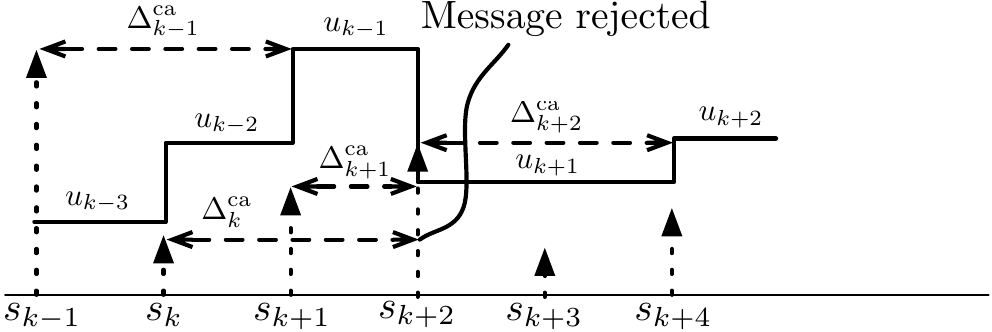}
\caption {Time-delays in the controller-to-actuator branch of the network with $\Delta^{\text{ca}}_k\in\{\tau,2\tau,3\tau\}$.}\label{reject}
\end{figure}

\subsection{Architecture of the symbolic controller}\label{controller}
A symbolic controller is a finite system that takes the observed states $x_k\in\R^n$ as inputs and produces as outputs the actions $u_k\in\mathsf{U}$ that need to be fed into the system $\Sigma$ in order to satisfy a given complex logical specification.  
We refer the interested readers to \cite{paulo} for the formal definition of symbolic controllers. 
Although for some LTL specifications (e.g. certain safety or reachability problems) it may be sufficient to consider only static controllers (i.e. without memory) \cite{girard4}, we do not limit our work by such an assumption and the proposed approach in this paper is indeed applicable to general LTL specifications \cite{katoen08}.
Due to the presence of a ZOH, from now on we assume that the set $\mathcal{U}$ contains only curves that are constant over intervals of length $\tau\in\R^+$ and take values in $\mathsf{U}$, i.e.:
\begin{align}
\label{input_set}
\mathcal{U}=\big\{&\upsilon:\R_0^+\to\mathsf{U}|\upsilon(t)=\upsilon((s-1)\tau), t\in[(s-1)\tau,s\tau[,s\in\N\big\}.
\end{align}
Correspondingly, one should update $U_\tau$ to $\mathcal{U}$ in \eqref{input_set} in the definition of $S_\tau(\Sigma)$ (cf. Section \ref{symbolic_model}).

Similar to what was assumed at the connection between controller and plant, 
we also consider possible occurrences of message rejection for the measurement data sent from the sensor to the symbolic controller. 
The symbolic controller uses $\widehat{x}_k$ as an input at the sampling times $s_k\Let k\tau$, where 
\begin{align}\label{new_state}
\widehat{x}_k=x_{k+\ell^k_*-N_{\max}^{\text{sc}}},
\end{align}
where $\ell^k_*\in\left[0;N^{\text{sc}}_{\max}-N^{\text{sc}}_{\min}\right]$ is defined as:
\begin{align}\label{switching_time1}
\ell^k_*=\lambda(\widetilde{N}_{N^{\text{sc}}_{\min}},\widetilde{N}_{N^{\text{sc}}_{\min}+1},\ldots,\widetilde{N}_{N^{\text{sc}}_{\max}}),
\end{align}
where $\widetilde{N}_\ell$, for $\ell\in\left[N^{\text{sc}}_{\min};N^{\text{sc}}_{\max}\right]$, is the delay suffered by the measurement packet sent $\ell$ samples ago, namely $\widetilde{N}_{N^{\text{sc}}_{\max}-i}=N^{\text{sc}}_{k-N^{\text{sc}}_{\max}+i}$ for any $i\in\left[0;N^{\text{sc}}_{\max}-N^{\text{sc}}_{\min}\right]$, and $\lambda$ is the function appearing in \eqref{switching_time}. 
Note that the expression for the input of the controller in \eqref{new_state} and \eqref{switching_time1} takes into account the possible out-of-order packet arrivals and message rejections. We again refer the interested readers to \cite{marieke,nathan} for more details on the proposed choice for $\ell^k_*$ \eqref{switching_time1}. Here, we assume that the symbolic controller applies its previously computed input value if it does not receive a concrete state measurement from the network, which may be the case for a small interval of time after $s_0$ due to the initialization of the NCS. 

\subsection{Describing NCS as metric systems}
As emphasized earlier, one of the main objectives of this work is to provide symbolic models for the overall NCS using symbolic models of their plants component and of the network characteristics. 
Specifically, we need to define a map taking an (in)finite system describing the plant and the minimum and maximum delays suffered in both the controller-to-actuator and the sensor-to-controller branches of the network as its inputs and providing, correspondingly, an (in)finite system describing the overall NCS as its output. 
Consider the map 
\begin{align}\label{map}
\mathcal{L}:\mathcal{T}(U,Y)\times\N_0^4\rightarrow\mathcal{T}(U,Y)
\end{align}
defined as the following: $\forall~\widetilde{N}_{\min},\widetilde{N}_{\max}\in\N_0$, where $\widetilde{N}_{\min}\leq \widetilde{N}_{\max}$, $\forall~\widehat{N}_{\min},\widehat{N}_{\max}\in\N_0$, where $\widehat{N}_{\min}\leq \widehat{N}_{\max}$, and $\forall~S_a=(X_a,X_{a0},U_a,\rTo_a,Y_a,H_a)\in\mathcal{T}(U_a,Y_a)$, we have $\mathcal{L}(S_a,\widetilde{N}_{\min},\widetilde{N}_{\max},\widehat{N}_{\min},\widehat{N}_{\max})=S_b\in\mathcal{T}(U_a,Y_a)$, where $S_b=(X_b,X_{b0},U_a,\rTo_b,Y_a, H_b)$ and
\begin{itemize}
\item $X_b=\left\{X_a\cup{q}\right\}^{\widetilde{N}_{\max}}\times U_a^{\widehat{N}_{\max}}\times[\widetilde{N}_{\min};\widetilde{N}_{\max}]^{\widetilde{N}_{\max}}\times[\widehat{N}_{\min};\widehat{N}_{\max}]^{\widehat{N}_{\max}}$, where $q$ is a dummy symbol;
\item $X_{b0}=\{(x_0,q,\ldots,q,u_0,\ldots,u_0,\widetilde{N}_{\max},\ldots,\widetilde{N}_{\max},\widehat{N}_{\max},\ldots,\widehat{N}_{\max})\,\,|\,\,x_0\in X_{a0},u_0\in{U}_a\}$;
\item $(x_1,\ldots,x_{\widetilde{N}_{\max}},u_1,\ldots,u_{\widehat{N}_{\max}},\widetilde{N}_1,\ldots,\widetilde{N}_{\widetilde N_{\max}},\widehat{N}_1,\ldots,\widehat{N}_{\widehat N_{\max}})\rTo^{u}_b(x',x_1,\ldots,x_{\widetilde N_{\max}-1},u,u_1,\ldots,\\u_{\widehat N_{\max}-1},\widetilde{N},\widetilde{N}_1,\ldots,\widetilde{N}_{\widetilde N_{\max}-1},\widehat{N},\widehat{N}_1,\ldots,\widehat{N}_{\widehat N_{\max}-1})$ for all $\widetilde{N}\in[\widetilde N_{\min};\widetilde N_{\max}]$ and all $\widehat{N}\in[\widehat N_{\min};\widehat N_{\max}]$ if there exists transition $x_1\rTo_{a}^{u_{\widehat N_{\max}-j^k_*}}x'$ in $S_a$ where $j^k_*=\lambda(\widehat{N}_{\widehat N_{\min}},\ldots,\widehat{N}_{\widehat N_{\max}})$, as defined in \eqref{switching_time}, and one of the
following holds (due to the initialization of the NCS):
\begin{itemize}
\item $x_{\widetilde N_{\max}-\ell^k_*}=q$, where $\ell^k_*=\lambda(\widetilde{N}_{\widetilde N_{\min}},\ldots,\widetilde{N}_{\widetilde N_{\max}})$, defined in \eqref{switching_time1}, and $u=u_1$;
\item $x_{\widetilde N_{\max}-\ell^k_*}\neq q$ and the choice of $u$ is free;
\end{itemize}
\item $H_b(x_1,\ldots,x_{\widetilde N_{\max}},u_1,\ldots,u_{\widehat N_{\max}},\widetilde{N}_1,\ldots,\widetilde{N}_{\widetilde N_{\max}},\widehat{N}_1,\ldots,\widehat{N}_{\widehat N_{\max}})=H_a(x_1)$ where 
with a slight abuse of notation, we assume that $H_a(q)\Let{q}$.
\end{itemize}

It can be readily seen that the system $S_b$ is (un)countable or symbolic if the system $S_a$ is (un)countable or symbolic, respectively. Although $S_a$ may be a deterministic system, $S_b$ is 
in general a nondeterministic system (if $\widetilde{N}_{\min}<\widetilde{N}_{\max}$ or $\widehat{N}_{\min}<\widehat{N}_{\max}$), since depending on the values of $\widetilde{N}$ or $\widehat{N}$, more than one $u$-successor of any state of $S_b$ may exist.

We assume additionally that the output set $Y_b$ is equipped with the same metric $\mathsf{d}_{Y_a}$, 
which is extended so that $\mathsf{d}_{Y_a}(H_a(x),H_a(q))=+\infty$ for any $x\in\R^n$ and $\mathsf{d}_{Y_a}(H_a(q),H_a(q))=0$. 

We have now all the ingredients to describe the NCS $\widetilde\Sigma$ as a metric system. 
Given $S_\tau(\Sigma)$ and the NCS $\widetilde\Sigma$, consider the metric system $S(\widetilde\Sigma):=(X,X_{0},U,\rTo,Y,H)$, capturing all the information contained in the NCS $\widetilde\Sigma$, given as $S(\widetilde\Sigma)=\mathcal{L}(S_\tau(\Sigma),N^{\text{sc}}_{\min},N^{\text{sc}}_{\max},N^{\text{ca}}_{\min},N^{\text{ca}}_{\max})$.

Note that the choice of the state space $X$ in $S(\widetilde\Sigma)$ allows us to keep track of an adequate number of measurements and control packets and the corresponding delays suffered by them, 
which is necessary and sufficient in order to consider out-of-order packet arrivals and message rejections as explained in detail in \cite{marieke,nathan}. The choice of the set of initial state $X_0$ keeps the initial input value $u_0$ in the ZOH till new control input values arrive. Moreover, assigning the maximum delay suffered by the dummy symbols ensures that those symbols will not take over an actual packet at the later iterations of the network.
The transition relation of $S(\widetilde\Sigma)$ captures in a nondeterministic fashion all the possible successors of a given state of $S(\widetilde\Sigma)$, 
based on all the possible ordering of measurements arriving to the controller, and of inputs arriving to the ZOH and ensures that the controller applies its previously computed input value if it does not receive any concrete state measurement from the network.
Let us also remark that the sets of states and inputs of $S(\widetilde\Sigma)$ are uncountable. 

\begin{remark}\label{rem_output}
Note that the output value of any state of $S(\widetilde\Sigma)$ is simply the output value of the state of the plant available at the sensors at times $s_k\Let k\tau$. We should highlight that the main role of output sets (resp. maps) in the definition of systems (cf. Definition \ref{system}) is to describe the set of atomic propositions (resp. state labeling) used in describing the specifications and, hence, used for the symbolic controller synthesis. We refer the interested readers to \cite[Chapter 5]{paulo} explaining controller synthesis schemes for some classes of specifications in which the output set plays a role; see in particular the discussion after the proof of Proposition 6.8 in \cite{paulo}. For the implementation (refinement) of symbolic controllers and their composition, one requires to deal with the states of systems rather than their outputs \cite[Proposition 8.7]{paulo}. We elaborate more on the symbolic controller synthesis and refinement in Section \ref{synthesis}.
\end{remark}

\section{Symbolic Models for NCS}\label{existence}
This section contains the main contributions of the paper. 
We show the existence and construction of symbolic models for NCS by using an existing symbolic model for the plant $\Sigma$, 
namely $S_\params(\Sigma)\Let(X_\params,X_{\params0},U_\params,\rTo_\params,Y_\params,H_\params)$.

Given the metric system $S_\params(\Sigma)$, define the new metric system $S_*(\widetilde\Sigma):=(X_*,X_{*0},U_*,\rTo_*,Y_*,H_*)$ as $S_*(\widetilde\Sigma)=\mathcal{L}(S_\params(\Sigma),N^{\text{sc}}_{\min},N^{\text{sc}}_{\max},N^{\text{ca}}_{\min},N^{\text{ca}}_{\max})$, where the map $\mathcal{L}$ is defined in \eqref{map}. 
System $S_*(\widetilde\Sigma)$ is constructed in the same way as $S(\widetilde\Sigma)$, 
but replacing continuous states, inputs, and the transition relation of $S_\tau(\Sigma)$, 
with the corresponding ones in $S_\params(\Sigma)$.  

We can now state the first pair of major technical results of this work, which are schematically represented in Figure \ref{NCS6}. 

\begin{figure}
\begin{center}
\includegraphics[width=7cm]{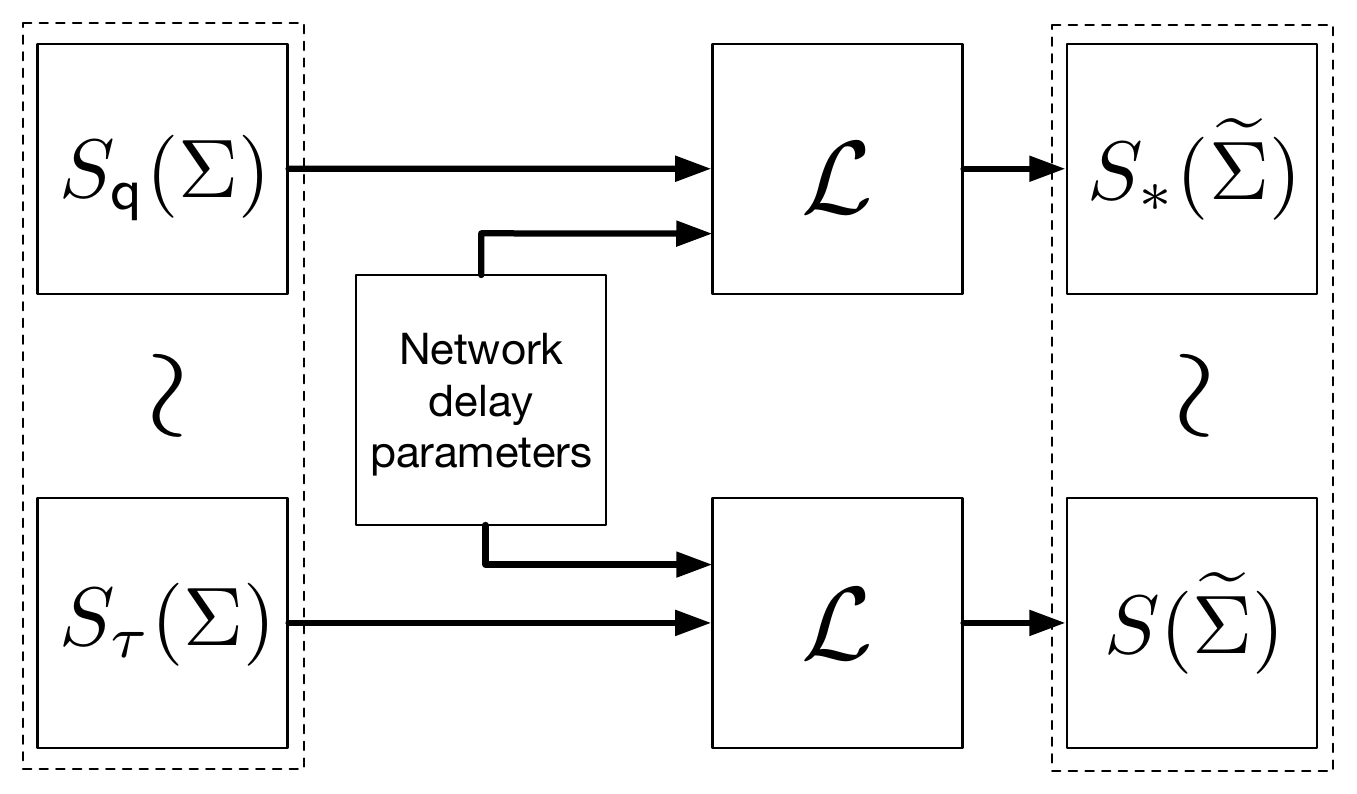}
\caption{The symbol $\sim$ represents any of the following relations: $_{\mathcal{S}}^{\varepsilon}{\succeq}$, $\preceq_{\mathcal{AS}}^{\varepsilon}$, and $\cong_{\mathcal{AS}}^{\varepsilon}$.}\label{NCS6}
\end{center}
\end{figure}

\begin{theorem}\label{main_theorem}
Consider an NCS $\widetilde\Sigma$ and suppose that there exists an abstraction $S_\params(\Sigma)$ such that $S_\params(\Sigma)\preceq_{\mathcal{AS}}^{\varepsilon}S_\tau(\Sigma)\preceq_{\mathcal{S}}^{\varepsilon}S_\params(\Sigma)$. 
Then we have \mbox{$S_{*}(\widetilde\Sigma)\preceq_{\mathcal{AS}}^{\varepsilon}S(\widetilde\Sigma)\preceq_{\mathcal{S}}^{\varepsilon}S_{*}(\widetilde\Sigma)$}.
\end{theorem}
The proof is provided in the Appendix.

\begin{corollary}\label{main_theorem1}
Consider an NCS $\widetilde\Sigma$ and suppose that there exists an abstraction $S_\params(\Sigma)$ such that $S_{\params}(\Sigma)\cong_{\mathcal{AS}}^{\varepsilon}S_\tau(\Sigma)$. 
Then we have $S_{*}(\widetilde\Sigma)\cong_{\mathcal{AS}}^{\varepsilon}S(\widetilde\Sigma)$.
\end{corollary}
The proof is provided in the Appendix.

\begin{remark}
As discussed earlier, one of the main advantages of the results proposed here in comparison with the ones in \cite{pola4,pola5} is that one can construct symbolic models for NCS using symbolic models obtained exclusively for the plant. Therefore, one can readily extend the proposed results to other classes of control systems for the plants, e.g. stochastic control systems, as long as there exist techniques to construct the corresponding symbolic models. For example, one can leverage the recently developed results in \cite{majid8}, \cite{majid10,ZTA1} (not requiring state-space gridding), and \cite{majid7} to construct symbolic models for classes of stochastic plants embedded in NCS.
\end{remark}

\subsection{Limited bandwidth}\label{band}
Assume that an abstraction $S_\params(\Sigma)$ exists such that $S_\params(\Sigma)\preceq_{\mathcal{AS}}^{\varepsilon}S_\tau(\Sigma)$ equipped with the alternating $\varepsilon$-approximate simulation relation $R$. From the formal definition of symbolic controllers in \cite{paulo} constructed based on $S_\params(\Sigma)$, one can readily verify the implicit presence of a static set-valued map (a.k.a quantizer map) $\varphi:X_\tau\rightarrow 2^{X_\params}$ inside the symbolic controllers, associating to each $x_\tau\in X_\tau$ a set of symbols in ${X_\params}$ as the following: $$\varphi(x_\tau)=\left\{x_\params\in X_\params\,\,|\,\,(x_\params,x_\tau)\in R\right\}.$$ Since the map $\varphi$ is static, one can shift this map towards the sensor in the NCS, as shown in Figure \ref{fig3}, without affecting any of the presented results. This means that in general a set of symbols, rather than only a \emph{quantized} one, needs to be sent over the sensor-to-controller branch of the network. Let us provide a simple example illustrating the problem that may raise if only one of the multiple possible symbols is sent instead of all of them. 

\begin{figure}
\centering
\begin{tabular}{c}
\includegraphics[width=8.5cm]{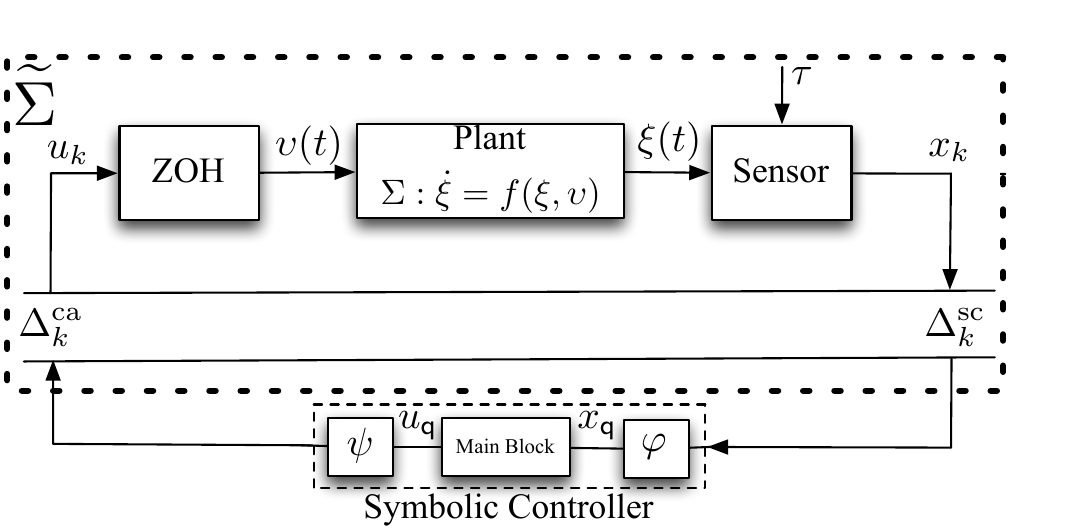}\\
\includegraphics[width=8.5cm]{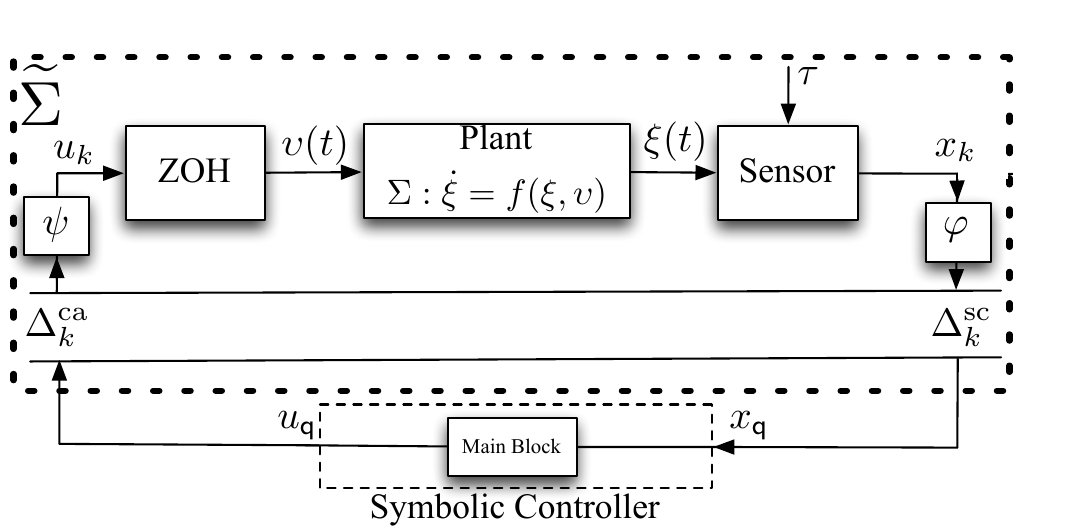}
\end{tabular}
\caption{Shifting maps $\varphi$ and $\psi$ for the symbolic controller to the other side of the communication network.}\label{fig3}
\end{figure}

\begin{example}
Consider the pair of finite systems in Figure \ref{fig33}, where the initial states are shown as targets of sourceless arrows and the lower part of the states are labeled with their output values. One can readily verify that $R=\left\{(\ol{x}_1,x_1),(\ol{x}_2,x_2),(\ol{x}_3,x_2)\right\}$ is an alternating 0-approximate simulation relation from $\ol{S}$ to $S$.
Therefore, $\varphi(x_1)=\{\ol{x}_1\}$ and $\varphi(x_2)=\{\ol{x}_2,\ol{x}_3\}$ is the associated ``quantization" map resulting from the relation $R$. Let us consider the new quantization map $\widetilde\varphi$ providing only one state of $\ol{S}$ for each state of $S$: $\widetilde\varphi(x_1)=\left\{\ol{x}_1\right\}$ and $\widetilde\varphi(x_2)=\{\ol{x}_3\}$. Consider the problem of synthesizing a controller enforcing the output of $\ol{S}$ to reach and stay at set $\{2\}$, namely a controller for the LTL specification $\Diamond \Box \{2\}$. There are infinitely many control sequences over $\ol{S}$ satisfying $\Diamond \Box \{2\}$, e.g. $\ol{u}_1\ol{u}_3\ol{u}_3\cdots$, $\ol{u}_2\ol{u}_2\ol{u}_1\ol{u}_3\ol{u}_3\cdots$, and $\ol{u}_2\ol{u}_2\ol{u}_2\ol{u}_2\ol{u}_1\ol{u}_3\ol{u}_3\cdots$. A possible ``static" controller enforcing the desired property could thus be obtained by restricting the set of inputs that the controller offers at each state of the abstracted plant, e.g. a map offering at $\bar{x}_1$ input $\bar{u}_1$, at $\bar{x}_3$ input $\bar{u}_2$, and at $\bar{x}_2$ input $\bar{u}_3$.  Using the new quantizer map $\widetilde\varphi$, and a controller consisting solely of the map in the previous sentence, however, does not allow us to distinguish between $\bar{x}_2$ and $\bar{x}_3$ and the refined control sequences over $S$ would result in $u_1u_2u_1u_2u_1u_2\cdots$. 
Such controller would result in the system satisfying infinitely often reaching $\{2\}$ on $S$, i.e. $\Box \Diamond  \{2\}$, rather than satisfying the requested specification $\Diamond \Box \{2\}$.  While this is a clearly concocted example for illustrative purposes, situations analogous to the one captured by this example arise in the construction of abstractions via notions of (alternating) approximate (bi)simulation (e.g.~\cite{majid}) in which some concrete states may be associated to several abstract states. For more details on this potential problem we refer the interested readers to~\cite{Matthias_Gunther}.  

\begin{figure}
\begin{center}
\begin{minipage}{0.4\linewidth}
\scalebox{0.8}{
\begin{tikzpicture}[shorten >=1pt,node distance=2.0cm,auto]
\tikzstyle{state}=[state with output]
\tikzstyle{every state}=[draw=brown!100,very thick,fill=brown!50,minimum size=1cm]
\node[state,initial,initial text=,initial where=above] (x1) {$\ol{x}_1$ \nodepart{lower} {$1$}};
\node[state] (x2) [right of=x1] {$\ol{x}_2$ \nodepart{lower} {$2$}};
\node[state] (x3) [below of=x1] {$\ol{x}_3$ \nodepart{lower} {$2$}};
\node[above left of=x1] {$\ol{S}:$};
\path[->] (x1) edge node {$\ol{u}_1$} (x2)
(x3) edge [bend left=35] node {$\ol{u}_2$} (x1)
(x1) edge [bend left=35] node {$\ol{u}_2$} (x3)
(x2) edge [loop right] node {$\ol{u}_3$} (x2);
\end{tikzpicture}
}
\end{minipage}\quad
\begin{minipage}{0.3\linewidth}
\vspace{-1.2cm}
\scalebox{0.8}{
\begin{tikzpicture}[shorten >=1pt,node distance=2.0cm,auto]
\tikzstyle{state}=[state with output]
\tikzstyle{every state}=[draw=brown!100,very thick,fill=brown!50,minimum size=1cm]
\node[state,initial,initial text=,initial where=above] (x1) {$x_1$ \nodepart{lower} {$1$}};
\node[state] (x2) [right of=x1] {$x_2$ \nodepart{lower} {$2$}};
\node[above left of=x1] {$S:$};
\path[->] (x1) edge [bend right=-35] node {$u_1$} (x2)
(x2) edge [bend left=35] node {$u_2$} (x1)
(x2) edge [loop right] node {$u_3$} (x2);
\end{tikzpicture}
}
\end{minipage}
\caption {Finite systems $\ol{S}$ and $S$.}\label{fig33}
\end{center}
\end{figure}
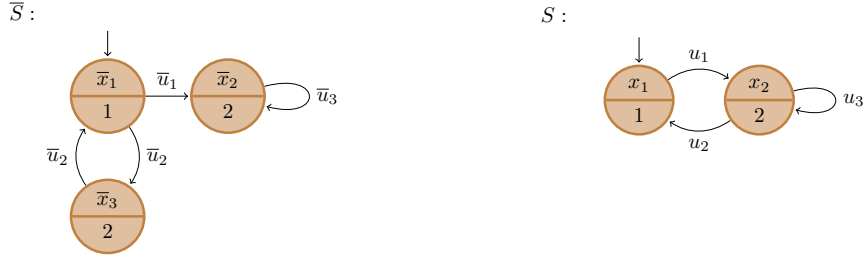
\end{example}

\begin{remark}\label{inform}
Unfortunately, the problem we just illustrated may arise in the constructions of \cite{pola4,pola5}. Based on the proposed symbolic abstractions in those works, the set-valued quantizer map $\varphi:\R^n\rightarrow 2^{[\R^n]_\eta}$ should be as follows: $$\varphi(x)=\left\{x_\params\in[\R^n]_\eta\,\,|\,\,\Vert x-x_\params\Vert\leq\varepsilon\right\},$$ for some given state-space quantization parameter $\eta\in\R^+$ and some precision $\varepsilon\in\R^+$, where $\eta<\varepsilon$; see \cite[equation (18)]{pola4} and \cite[equation (5)]{pola5}. However, \cite{pola4,pola5} use the map $\widetilde\varphi:x\rightarrow[x]_\eta$, where $[x]_\eta\in\left[\R^n\right]_\eta$ associates to every $x\in \R^n$ just one quantized state $[x]_\eta\in [\R^n]_\eta$, such that $\left\Vert x-[x]_\eta\right\Vert\leq\eta/2$. For the case of deterministic quantizers (no measurement error), this problem can be readily avoided if the proposed alternating approximate simulation relations in those papers were directly defined over quantized states as proposed in \cite{girard6}. For the case of nondeterministic quantizers, either one should send a set of symbols to the controllers, as discussed in the beginning of Subsection \ref{band}, or one should resort to feedback refinement relations \cite{Matthias_Gunther} (cf. Remark \ref{feed_refine}) and send only one symbol to the controller.
\end{remark}


Similarly, a quantization map $\ol{\psi}:X_\params\times X_\tau \times U_\params\rightarrow \mathsf{U}$ is implicitly contained in the symbolic controllers, associating to each symbol $u_\params\in U_\params(x_\params)$ generated by the controller an input $u\in U_\tau(x_\tau)$ for some $(x_\params,x_\tau)\in R$. Unfortunately, the quantization map $\ol\psi$ requires the knowledge of the state of the plant just before the controller. Therefore, one cannot easily shift this map towards the actuator (ZOH) in the NCS scheme. In order to solve this issue, one can simply assume that the set $\mathsf{U}$ is finite and $U_\params=\mathsf{U}$ and 
adjust condition (iii) in Definition \ref{AASR} as: 
\begin{itemize}
\item[(iii)] for every $(x_\params,x_\tau)\in R$, every $u_\params\in U_\params\left(x_\params\right)$, and every $x'_\tau\in\textbf{Post}_{u_\params}(x_\tau)$ there exists $x'_\params\in\mathbf{Post}_{u_\params}(x_\params)$ satisfying $(x'_\params,x'_\tau)\in R$,
\end{itemize}  
so that only abstractions $S_\params(\Sigma)$ satisfying $S_\params(\Sigma)\preceq_{\mathcal{AS}}^{\varepsilon}S_\tau(\Sigma)$ with the new condition (iii) are admitted in our scheme.
These modifications simply imply that for each symbolic input $u_\params$ in $S_\params(\Sigma)$ one should apply the same input to $S_\tau(\Sigma)$. Note that we abused notation by identifying $u_\params$ with the constant input curve with domain $[0,\tau[$ and value $u_\params$. With this adjustments, one has a new quantizer map 
$\psi=1_{U_\params}$, which is static and can be shifted towards the actuator (ZOH) in the NCS, as shown in Figure \ref{fig3}. Note that the proposed abstractions in \cite{girard2,majid,majid10,ZTA1,majid7,majid8} satisfy this new condition in Definition \ref{AASR} by simply taking $U_\params=\mathsf{U}$ in those results. In general this is a rather natural assumption to be taken as in practice one usually considers a finite set of inputs available and constructs abstractions accordingly. We emphasize that the results in Theorem \ref{main_theorem} and Corollary \ref{main_theorem1} still hold with this modification on condition (iii) in Definition \ref{AASR}.

\begin{remark}\label{feed_refine}
Observe that one can use the recently developed notion of feedback refinement relations introduced in \cite{Matthias_Gunther} in order to establish the relation between the concrete systems and their symbolic models. This new relation resolves both issues explained in the previous paragraphs: i) the refined controller only requires the quantized state information of the concrete
system; ii) the abstraction does not need to be used as a building block inside the refined controller
and, consequently, a smaller amount of memory is required. We refer the interested readers to \cite{KRZ} showing that the proposed map $\mathcal{L}$ in \eqref{map} also preserves the feedback refinement relations and that similar results as in Theorem \ref{main_theorem} hold for this new relation as well.
\end{remark}

\section{Symbolic Controller Synthesis and Refinements}\label{synthesis}
\subsection{Symbolic controller synthesis}\label{synthesis1}
Although the main contribution of the paper is on the construction of symbolic models for NCS with some non-idealities, the provided abstractions are amenable to any off the shelf symbolic controller synthesis toolbox such as \texttt{SCOTS} \cite{scots} and \texttt{Slugs} \cite{slugs}. To further elaborate on this, let us consider the following example. Let $A\subset\mathbb{R}^n$ be a compact set. Consider a safety problem, formulated as the satisfaction of the LTL formula\footnote{We refer the interested readers to \cite{katoen08} for the formal semantics of the temporal formula $\Box \varphi_A$ expressing the safety property over the set $A$.} $\Box \varphi_A$, where $\varphi_A$ is a label (or atomic proposition) characterizing the set $A$. The goal is to synthesize a controller enforcing $\Box \varphi_A$ over the output of the plant, available at the sensors before the network. To do so, we first construct a discrete controller enforcing $\Box \varphi_A$ over the output of $S_*(\widetilde\Sigma)=(X_*,X_{*0},U_*,\rTo_*,Y_*,H_*)$. Whenever $Y_*\neq X_*$ and $H_*\neq 1_{X_*}$, it suffices to consider a new safe set $\widehat A\subseteq X_*$ defined by $\widehat A=\{\mathsf{x}_*\in X_*\,\,|\,\,H_*(\mathsf{x}_*)\in A\}$. Now, one can apply Theorem 6.6 in \cite{paulo} to auxiliary system $\widehat S_*(\widetilde\Sigma)=(X_*,X_{*0},U_*,\rTo_*,X_*,1_{X_*})$ and the specification set $\widehat A$ to synthesize a discrete controller enforcing $\Box \varphi_A$ over the output of $S_*(\widetilde\Sigma)$. The main subtlety here is in the refinement of the constructed discrete controller enforcing $\Box \varphi_A$ over the output of the plant which requires the whole state tuple $\mathsf{x}_*$ of $S_*(\widetilde\Sigma)$ while only one of the elements of the tuple is available based on the packet arrived before the controller. We elaborate on the refinement of symbolic controllers in the next subsection and propose a class of NCS in which the whole state tuple $\mathsf{x}_*$ of $S_*(\widetilde\Sigma)$ can be recovered inside the controllers.

\subsection{Symbolic controller refinement}\label{refine}
In order to refine the synthesized symbolic controllers in our setup, we target a class of NCS where the upper and lower bounds of the delays are equal at each channel. This implies that all packets suffer the same delay (i.e. $\widetilde N_k = N^{\text{sc}}_{\min} = N^{\text{sc}}_{\max}$ and $\widehat N_k = N^{\text{ca}}_{\min} = N^{\text{ca}}_{\max}$ for any $k\in\N_0$) in each channel. This can be readily achieved by performing extra prolongation (if needed) of the delays suffered already by the packets. For the sensor-to-controller channel, this can be readily done inside the controller. The controller needs to have a buffer to hold arriving packets and keep them in the buffer until their delays reach the maximum. For the controller-to-actuator channel, the same needs to be implemented inside the ZOH. Therefore, in this setting, state (resp. input) packets are allowed to have any delay (not necessarily integer multiples of the sampling time) between 0 and $ N^{\text{sc}}_{\max}$ (resp. $ N^{\text{ca}}_{\max}$) where $ N^{\text{sc}}_{\max}$ and $ N^{\text{ca}}_{\max}$ are integer multiples of the sampling time.

In this special class of NCS, the information contained in the NCS $\widetilde\Sigma$ is captured by the metric system $S'(\widetilde \Sigma):=\mathcal{L}(S_\tau(\Sigma), N^{\text{sc}}_{\max},N^{\text{sc}}_{\max},N^{\text{ca}}_{\max},N^{\text{ca}}_{\max})$. We also denote by $S_*'(\widetilde \Sigma):=\mathcal{L}(S_\params(\Sigma), N^{\text{sc}}_{\max},N^{\text{sc}}_{\max},N^{\text{ca}}_{\max},N^{\text{ca}}_{\max})$ the corresponding symbolic model of $S'(\widetilde \Sigma)$. Recall that $S_*(\widetilde \Sigma)$ denotes the symbolic model of NCS without the prolongation of delays suffered by packets in both channels of the network. Here, we provide a brief comparison between $S_*'(\widetilde \Sigma)$ and $S_*(\widetilde \Sigma)$: 1) $S_*'(\widetilde \Sigma)$ has no non-determinism caused by different delay possibilities in comparison with $S_*(\widetilde \Sigma)$. This results in a smaller transition relation making the controller synthesis less complex; 2)  $S_*'(\widetilde \Sigma)$ is less conservative in comparison with $S_*(\widetilde \Sigma)$ in terms of the existence of symbolic controllers satisfying some given logic specifications. We elaborate more on this in a lemma later;  3) In terms of actual implementation, the controllers designed for $S'(\widetilde \Sigma)$ may be more complex than those for $S(\widetilde \Sigma)$ because they need to have a buffer to hold arriving packets till they reach the required maximum delay, the same needs to be implemented for the ZOH. 

\begin{lemma}\label{less}
Consider a symbolic model $S_a$ and $\widetilde N_{\min},\widetilde N_{\max},\widehat N_{\min},\widehat N_{\max}\in\N_0$, where $\widetilde N_{\min}\leq\widetilde N_{\max}$ and $\widehat N_{\min}\leq\widehat N_{\max}$. We have \mbox{$S_{*}\preceq_{\mathcal{AS}}^{0}S'_{*}$}, where $S_{*}:=\mathcal{L}(S_a, \widetilde N_{\min},\widetilde N_{\max},\widehat N_{\min},\widehat N_{\max})$ and $S'_{*}:=\mathcal{L}(S_a, \widetilde N_{\max},\widetilde N_{\max},\\\widehat N_{\max},\widehat N_{\max})$.
\end{lemma}
The proof is provided in the Appendix. The result in Lemma \ref{less} implies that if there exists a symbolic controller enforcing some complex specifications over $S_{*}$, then there exists a symbolic controller enforcing the same complex specifications over $S'_{*}$ which confirms item 2) in the above comparison between $S_*'(\widetilde \Sigma)$ and $S_*(\widetilde \Sigma)$.

Finally, in order to refine the constructed symbolic controllers in closed loop fashion, one needs to have the symbolic state tuple of the form: $$(x_{*1},\ldots,x_{*N^{\text{sc}}_{\max}},u_1,\ldots,u_{N^{\text{ca}}_{\max}},N^{\text{sc}}_{\max},\ldots,N^{\text{sc}}_{\max},N^{\text{ca}}_{\max},\ldots,N^{\text{ca}}_{\max}).$$The controller already knows what control-inputs has generated during the $N_{\max}:=N^{\text{ca}}_{\max}+N^{\text{sc}}_{\max}-1$ previous sampling times (i.e. $u_1,\ldots,u_{N_{\max}}$). Hence, it just needs to store them in a buffer. The first $N^{\text{ca}}_{\max}$ control inputs (i.e. $u_1,\ldots,u_{N^{\text{ca}}_{\max}}$) will be used directly in the symbolic state tuple and the rest for the construction of states $x_{*1},\ldots,x_{*(N^{\text{sc}}_{\max}-1)}$. Now consider two different cases. Case 1: we assume that the symbolic model of the plant (i.e. $S_\params(\Sigma)$) is deterministic (cf. the example section). The controller gets states $x_{*N^{\text{sc}}_{\max}}$ using the current measurement packet (i.e. $x_{N^{\text{sc}}_{\max}}$) and the relation between $S_\tau(\Sigma)$ and $S_\params(\Sigma)$. Using $x_{*N^{\text{sc}}_{\max}}$, previously generated control-inputs (i.e. $u_{N^{\text{ca}}_{\max}+1},\ldots,u_{N^{\text{ca}}_{\max}+N^{\text{sc}}_{\max}-1}$), and symbolic model $S_\params(\Sigma)$, the controller can construct other symbolic state information as the following: $x_{*1} =\textbf{Post}_{u_{N^{\text{ca}}_{\max}+1}}(x_{*2})$, $x_{*2} =\textbf{Post}_{u_{N^{\text{ca}}_{\max}+2}}(x_{*3})$, $\ldots$, and $x_{*(N^{\text{sc}}_{\max}-1)} =\textbf{Post}_{u_{N^{\text{ca}}_{\max}+N^{\text{sc}}_{\max}-1}}(x_{*N^{\text{sc}}_{\max}})$. Case 2: we assume that the controller has access to the current state measurement of the plant (i.e. $x_{N^{\text{sc}}_{\max}}$) and the model of the plant. Here, the controller can construct all the state measurements still traveling inside the sensor-to-controller channel up to the current state of the plant (i.e. $x_1,\ldots,x_{N^{\text{sc}}_{\max}-1}$) using the current packet it receives (i.e. $x_{N^{\text{sc}}_{\max}}$), previously generated control-inputs (i.e. $u_{N^{\text{ca}}_{\max}+1},\ldots,u_{N^{\text{ca}}_{\max}+N^{\text{sc}}_{\max}-1}$), and the model of the plant: $x_1 =\xi_{x_2u_{N^{\text{ca}}_{\max}+1}}(\tau)$, $x_2 =\xi_{x_3u_{N^{\text{ca}}_{\max}+2}}(\tau)$, $\ldots$, and $x_{N^{\text{sc}}_{\max}-1} =\xi_{x_{N^{\text{sc}}_{\max}}u_{N^{\text{ca}}_{\max}+N^{\text{sc}}_{\max}-1}}(\tau)$  (solving the differential equation, possibly numerically, online). Therefore, using the relation between $S_\tau(\Sigma)$ and $S_\params(\Sigma)$ and $x_{1},\ldots,x_{N^{\text{sc}}_{\max}}$, symbolic states $x_{*1},\ldots,x_{*N^{\text{sc}}_{\max}}$ are constructed inside the controller.

\begin{remark}
One can use a quantized version of $x_{N^{\text{sc}}_{\max}}$ rather than itself in Case 2 above to construct symbolic states $x_{*1},\ldots,x_{*N^{\text{sc}}_{\max}}$ of (not necessarily deterministic) $S_\params(\Sigma)$ inside the controller. We can use a quantizer with appropriately chosen precision based on the abstraction precision $\varepsilon$, the Lipschitz constant $Z$ in Definition \ref{Def_control_sys}, and the proposed techniques in \cite[Subsection VI-B]{Matthias_Gunther} to construct symbolic states $x_{*1},\ldots,x_{*(N^{\text{sc}}_{\max}-1)}$ using the model of the plant. On the other hand, one can try to synthesize symbolic controllers with partial information (see e.g. \cite{chatterjee}) $(x_{*N^{\text{sc}}_{\max}},u_1,\ldots,u_{N^{\text{ca}}_{\max}},N^{\text{sc}}_{\max},\ldots,N^{\text{sc}}_{\max},N^{\text{ca}}_{\max},\ldots,N^{\text{ca}}_{\max})$ which is left as object of future research. Remark that the computational complexity of synthesis with partial information is usually much larger than the synthesis with full state information \cite{chatterjee}. Therefore, there is a trade off between having simpler controller synthesis scheme (cf. Subsection \ref{synthesis1}) amenable to any off the shelf synthesis toolbox but more complex refinement scheme (cf. Subsection \ref{refine}) or having more complex controller synthesis scheme (see e.g. \cite{chatterjee}) not necessarily tractable using off the shelf synthesis toolbox but simpler refinement procedure.
\end{remark}

\section{Space Complexity Analysis}

We compare the results provided here with those in \cite{pola4,pola5} in terms of the size of the obtained symbolic models. For the sake of a fair comparison, assume that we use also a grid-based symbolic abstraction for the plant $\Sigma$ using the same sampling time and quantization parameters as the ones in \cite{pola4,pola5}. Note that the provided comparison may not be complete still, because we do not need any requirement on the symbolic controller while in \cite{pola4,pola5} it is assumed that the symbolic controllers are static. By assuming that we are only interested in the dynamics of $\Sigma$ on a compact set $\mathsf{D}\subset\R^n$, the cardinality of the set of states of the symbolic models  provided in \cite{pola4,pola5}, is: 
\begin{align}\notag
\left\vert X_\star\right\vert=\displaystyle\sum_{i\in\left\{\{1\}\cup[N_{\min};N_{\max}]\right\}}\left\vert\left[\mathsf{D}\right]_{\eta}\right\vert^i,
\end{align}
where $N_{\min}=N^{\text{sc}}_{\min}+N^{\text{ca}}_{\min}$, $N_{\max}=N^{\text{sc}}_{\max}+N^{\text{ca}}_{\max}$, and $\left[\mathsf{D}\right]_{\eta}=\mathsf{D}\cap\left[\R^n\right]_\eta$ for some quantization parameters $\eta\in\R^+$.

Meanwhile, the size of the set of states for the abstractions provided by Theorem \ref{main_theorem} and Corollary \ref{main_theorem1}, is at most: 
\begin{align}\notag
\left\vert X_*\right\vert=&\left(\left\vert\left[\mathsf{D}\right]_{\eta}\right\vert+1\right)^{N^{\text{sc}}_{\max}}\cdot\left\vert\left[\mathsf{U}\right]_\mu\right\vert^{N^{\text{ca}}_{\max}}\cdot\left(N^{\text{sc}}_{\max}-N^{\text{sc}}_{\min}+1\right)^{N^{\text{sc}}_{\max}}\cdot\left(N^{\text{ca}}_{\max}-N^{\text{ca}}_{\min}+1\right)^{N^{\text{ca}}_{\max}},
\end{align}
where $\left[\mathsf{U}\right]_\mu=\mathsf{U}\cap\left[\R^m\right]_\mu$ for some quantization parameters $\mu\in\R^+$. 
Note that there may exist some states of $X_*$ that are not reachable from any of the initial states $x_{*0}\in X_{*0}$ due to the combination of the delays in both channels of the network and, hence, one can exclude them from the set of states $X_*$ without loss of generality. Therefore, the actual size of the state set $X_*$ may be less than the aforementioned computed ones.

One can easily verify that the size of the symbolic models proposed in \cite{pola4,pola5} is at most:
\begin{align}\label{size}
\left\vert S_\star(\widetilde\Sigma)\right\vert=&\left\vert X_\star\right\vert\cdot\left\vert[\mathsf{U}]_\mu\right\vert\cdot\left(N_{\max}-N_{\min}+1\right)\cdot{K}\\\notag=&\Big(\displaystyle\sum_{i\in\left\{\{1\}\cup[N_{\min};N_{\max}]\right\}}\left\vert\left[\mathsf{D}\right]_{\eta}\right\vert^i\Big)\cdot\left\vert[\mathsf{U}]_\mu\right\vert\cdot\left(N_{\max}-N_{\min}+1\right)\cdot{K},
\end{align}
where $K$ is the maximum number of $u$-successors of any state of the symbolic model $S_\params(\Sigma)$ for $u\in[\mathsf{U}]_\mu$. Note that with the results proposed in \cite{girard2} one has $K=1$ because $S_\params(\Sigma)$ is a deterministic system, while with the ones proposed in \cite{majid} one has $K\geq1$ because $S_\params(\Sigma)$ is a nondeterministic system and the value of $K$ depends on the functions $\beta$ and $\gamma$ in \eqref{delta_FC} -- see \cite{majid} for more details. The size of the symbolic models provided in this paper is at most:
\begin{align}\label{size1}
\left\vert S_*(\widetilde\Sigma)\right\vert&=\left\vert X_*\right\vert\cdot\left\vert[\mathsf{U}]_\mu\right\vert\cdot\left(N^{\text{sc}}_{\max}-N^{\text{sc}}_{\min}+1\right)\cdot\left(N^{\text{ca}}_{\max}-N^{\text{ca}}_{\min}+1\right).K\\\notag&=\left(\left\vert\left[\mathsf{D}\right]_{\eta}\right\vert+1\right)^{N^{\text{sc}}_{\max}}\cdot\left\vert\left[\mathsf{U}\right]_\mu\right\vert^{N^{\text{ca}}_{\max}+1}\cdot\left(N^{\text{sc}}_{\max}-N^{\text{sc}}_{\min}+1\right)^{N^{\text{sc}}_{\max}+1}\cdot\left(N^{\text{ca}}_{\max}-N^{\text{ca}}_{\min}+1\right)^{N^{\text{ca}}_{\max}+1}\cdot{K},
\end{align}
with the same $K$ as in \eqref{size}.
The symbolic model $S_*(\widetilde\Sigma)$ can have a smaller size for some large values of $N_{\max}$ and for $\left\vert\left[\mathsf{D}\right]_{\eta}\right\vert>>\left\vert\left[\mathsf{U}\right]_\mu\right\vert$, as depicted in Figure \ref{fig4} (upper panel) by fixing $N^{\text{ca}}_{\max}=N^{\text{sc}}_{\max}=6$ and $N^{\text{ca}}_{\min}=N^{\text{sc}}_{\min}=1$. On the other hand, the symbolic model $S_\star(\widetilde\Sigma)$ can have a smaller size for some large values of $\left\vert\left[\mathsf{U}\right]_\mu\right\vert$ and of $N^{\text{ca}}_{\max}-N^{\text{ca}}_{\min}$ (or $N^{\text{sc}}_{\max}-N^{\text{sc}}_{\min}$), as depicted in Figure \ref{fig4} (lower panel) by fixing $\vert\left[\mathsf{D}\right]_{\eta}\vert=10^7$.

\begin{figure}
\begin{center}
\includegraphics[width=10.0cm]{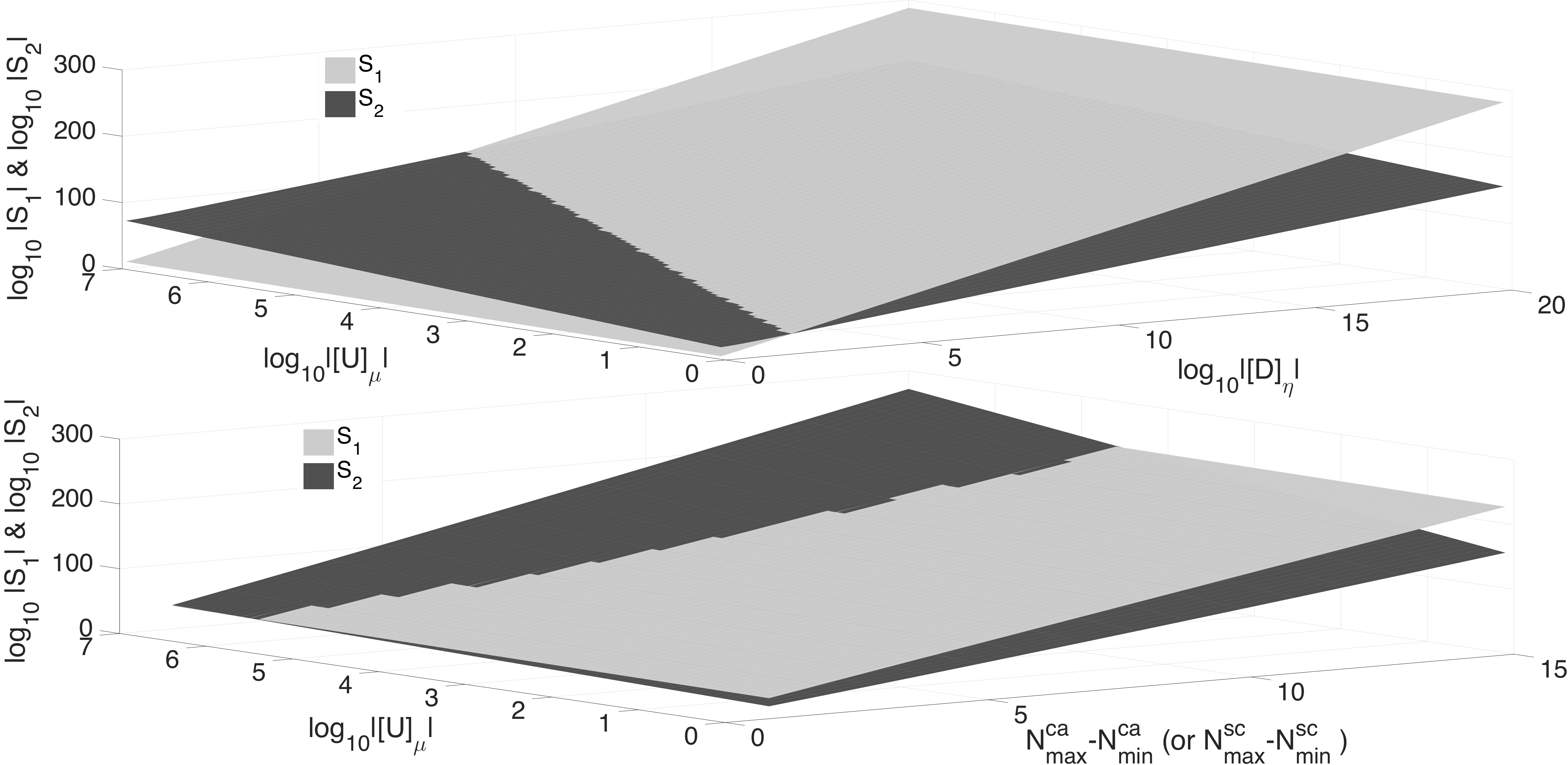}
\caption{Upper panel: sizes of $S_\star(\widetilde\Sigma)$ and $S_*(\widetilde\Sigma)$ for different values of $\vert\left[\mathsf{D}\right]_{\eta}\vert$ and $\vert\left[\mathsf{U}\right]_\mu\vert$, where $N^{\text{ca}}_{\max}=N^{\text{sc}}_{\max}=6$, $N^{\text{ca}}_{\min}=N^{\text{sc}}_{\min}=1$, and $S_1=S_\star(\widetilde\Sigma)$ and $S_2=S_*(\widetilde\Sigma)$. Lower panel: sizes of $S_\star(\widetilde\Sigma)$ and $S_*(\widetilde\Sigma)$ for different values of $\vert\left[\mathsf{U}\right]_\mu\vert$ and of $N^{\text{ca}}_{\max}-N^{\text{ca}}_{\min}$ (or $N^{\text{sc}}_{\max}-N^{\text{sc}}_{\min}$), where $\vert\left[\mathsf{D}\right]_{\eta}\vert=10^7$.}\label{fig4}
\end{center}
\end{figure}

Note that in the special case when $N^{\text{sc}}_{\max}=N^{\text{sc}}_{\min}=1$, the dummy symbol $q$ is not necessary in the definition of $X_*$, hence: 
\begin{align}
\label{num_state}
&\left\vert X_*\right\vert=\left\vert\left[\mathsf{D}\right]_{\eta}\right\vert\cdot\left\vert\left[\mathsf{U}\right]_\mu\right\vert^{N^{\text{ca}}_{\max}}\cdot\left(N^{\text{ca}}_{\max}-N^{\text{ca}}_{\min}+1\right)^{N^{\text{ca}}_{\max}}.
\end{align}
%

\begin{remark}
In \cite[Remark 5.2]{pola4} the authors suggest a more concise representation for their proposed finite abstractions of NCS, in order to reduce the space complexity. However, this representation is only applicable if the plant $\Sigma$ is $\delta$-ISS. Hence, for general classes of plants $\Sigma$ in the NCS, the approach proposed in this work can be more appropriate in terms of the size of the abstractions, particularly for large values of $N_{\max}$ and for $\left\vert\left[\mathsf{D}\right]_{\eta}\right\vert>>\left\vert\left[\mathsf{U}\right]_\mu\right\vert$.
\end{remark}

\begin{remark}
One can readily see in the example section that the computation time and memory required for computing symbolic abstractions of NCSs using the proposed method here are several orders of magnitude smaller than those required using techniques in \cite{pola4,pola5}. The main reason for this is because modular construction of abstractions as proposed in this paper is highly favored by the binary decision diagram (BDD) data structure which compactly represents both sets of states
and the transition relation between these states.
\end{remark}

\section{Example}\label{example} 
\begin{table*}[]	
	\small
	\centering
	\caption{Results for constructing symbolic models of NCS using symbolic models of their plants.}
	\label{TBL_MODEL_CONS_DETAILS}
	\scalebox{0.73}{
		\begin{tabular}{lll|llllllllll}
			\hline
			Case Study & $\vert S_\params(\Sigma)\vert$ &								& (2,2) 	  & (2,3)  	    & (2,4) 	  & (2,5)		& (3,2) 	  & (3,3)  	    & (3,4) 	  & (3,5) 	    & (4,2) 	  & (4,3) \\ \hline
			\textsf{SM} 	 & 26      		   & $\vert S_*(\widetilde\Sigma)\vert$ 	& 214   	  & 422    	    & 838   	  & 1670		& 430   	  & 846   	    & 1678  	  & 3342  	    & 862   	  & 1694  		\\
							 &         		   & Time (sec)	    						& $< 0.1$ 	  & $< 0.1$  	& $< 0.1$ 	  & $< 0.1$		& $< 0.1$ 	  & $< 0.1$	    & $< 0.1$ 	  & $< 0.1$     & $< 0.1$ 	  & $< 0.1$  	\\
							 &         		   & Memory (KB)    						& 1.9         & 2.9         & 1.2         & 1.2         & 3.5         & 1.6         & 1.6         & 1.6         & 1.9         & 1.9         \\ \cline{3-13}
			\textsf{DI} 	 & 2039    		   & $\vert S_*(\widetilde\Sigma)\vert$  	& 170272 	  & 681088 	    & 2.7$\tz{6}$ & 1.1$\tz{7}$ & 900384 	  & 3.6$\tz{6}$ & 1.4$\tz{7}$ & 5.8$\tz{7}$ & 4.8$\tz{6}$ & 1.9$\tz{7}$ \\
							 &         		   & Time (sec)	    						& $< 1$ 	  & $< 1$  	    & $< 1$ 	  & $< 1$		& $< 1$ 	  & $< 1$  	    & $< 1$ 	  & $< 1$ 	    & $< 1$ 	  & $< 1$  	    \\
							 &         		   & Memory (KB)    						& 2.0         & 2.4         & 3.1         & 2.9         & 3.0         & 2.9         & 3.1         & 3.2         & 5.2         & 4.3         \\ \cline{3-13}
			\textsf{Robot} 	 & 29280  		   & $\vert S_*(\widetilde\Sigma)\vert$  	& 6.4$\tz{7}$ & 1.0$\tz{9}$ & 1.6$\tz{10}$& 2.6$\tz{11}$& 5.6$\tz{8}$ & 9.0$\tz{9}$ & 3.4$\tz{11}$& 2.3$\tz{12}$& 4.9$\tz{9}$ & 7.8$\tz{10}$\\
							 &         		   & Time (sec)	    						& $< 1$ 	  & $< 1$  	    & $< 1$ 	  & $< 1$       & $< 1$       & $< 1$       & $< 1$       & $< 1$       & 1.4         & 1.6         \\
							 &         		   & Memory (KB)    						& 15          & 14          & 17          & 16          & 16          & 21          & 22          & 19          & 35          & 33          \\ \cline{3-13}							 
			\textsf{Vehicle1}& $9.1\times10^6$ & $\vert S_*(\widetilde\Sigma)\vert$  	& 2.4$\tz{12}$& 1.5$\tz{14}$& 9.7$\tz{15}$& 6.2$\tz{17}$& 1.5$\tz{14}$& 9.8$\tz{15}$& 6.3$\tz{17}$& 4.0$\tz{19}$& 1.0$\tz{16}$& 6.4$\tz{17}$\\
							 &         		   & Time (sec)	    						& 70 	  	  & 129  	 	& 89    	  & 107 		& 5587 	      & 944   	    & 1598 	      & 1705 	    & 7399  	  & 6182        \\
							 &         		   & Memory (KB) 							& 734         & 992         & 961         & 881         & 7577.6      & 5529.6      & 8294.4      & 7782.4      & 9932.8      & 9523.2  	\\ \cline{3-13}
			\textsf{Vehicle2}& $9.9\times10^6$ & $\vert S_*(\widetilde\Sigma)\vert$  	& 2.7$\tz{12}$& 1.7$\tz{14}$& 1.1$\tz{16}$& 7.0$\tz{17}$& 1.8$\tz{14}$& 1.2$\tz{16}$& 7.4$\tz{17}$& 4.6$\tz{19}$& 1.2$\tz{16}$& 7.91$\tz{17}$\\
							 &         		   & Time (sec)	    						& 66.8   	  & 104.6  		& 107   	  & 61.8		& 833   	  & 781  	    & 1247 	  	  & 4363 	    & 5137   	  & 8462	    \\
							 &         		   & Memory (KB) 							& 826         & 683         & 733         & 709         & 5222.4      & 4710.4      & 4608        & 11059,2     & 8396.8      & 8806.4      \\ \cline{3-13}
			\textsf{Vehicle3}& $1.89\times10^7$& $\vert S_*(\widetilde\Sigma)\vert$  	& 4.2$\tz{12}$& 2.7$\tz{14}$& 1.7$\tz{16}$& 1.1$\tz{18}$& 2.3$\tz{14}$& 1.4$\tz{16}$& 9.3$\tz{17}$& 5.9$\tz{19}$& 1.3$\tz{16}$& 7.94$\tz{17}$\\
							 &         		   & Time (sec)	    						& 273.3 	  & 285  		& 238.4 	  & 173 		& 22344 	  & 54919  	    & 27667 	  & 36467 	    & 39065 	  & 145390  	\\
							 &         		   & Memory (KB) 							& 1638.4      & 1945.6      & 1843.2      & 1945.6      & 23040       & 40652.8     & 30208       & 22425.6     & 21094.4     & 36556.3		\\
			\hline											 			 			
		\end{tabular}
	}
\end{table*}
In this section, we present some case studies where we construct symbolic models of NCS from the symbolic models of the plants inside them. We consider the setup presented in Section \ref{synthesis} in order to refine the constructed symbolic controllers in closed loop fashion. First, we present results for the construction of symbolic models of the NCS for several systems. Then, we provide an example where a dynamic controller is synthesized using the derived symbolic model of the NCS. The synthesized controller is simulated in closed loop fashion using both \texttt{MATLAB} and \texttt{OMNeT++} \cite{omnet++}. The computation
of the abstract systems $S'_*(\widetilde \Sigma)$ (cd. Subsection \ref{refine}) and the symbolic controllers have been implemented by
the software tool \texttt{SENSE} \cite{sense}.
\subsection{Symbolic models of NCS from the ones of plants in them}
We use the tool \texttt{SCOTS} \cite{scots} to construct symbolic models of the plants which are stored as BDD objects. The BDD objects are fed as inputs to the tool \texttt{SENSE} along with NCS delay bounds to construct symbolic models of NCS. Notice that the tool \texttt{SENSE} constructs symbolic models of NCS directly by operating with BDD objects of the symbolic models of the plants. This results in a large reduction in the computation time in comparison with constructing them from scratch which is the case using the techniques proposed in \cite{pola4,pola5} (cf. see later for a comparison for some of the case studies).
Table \ref{TBL_MODEL_CONS_DETAILS} summarizes the results for different network delay configurations. 
Six case studies are considered. For each case study, we show the size of the symbolic model of the plant. For different network delay configurations $(N_{\max}^{\text{sc}},N_{\max}^{\text{ca}})$, we show the size of the symbolic models of NCS, the time in seconds required to construct them, and the memory in KB used to store them.
First, we consider an already given symbolic model of the plant (denoted by \textsf{SM}) consisting of 13 states and 26 transitions. Then, we consider a plant as a double integrator (denoted by \textsf{DI}) inside an NCS where its dynamic is given by:
\begin{equation}
\nonumber
\Sigma:\left\{\dot{\xi} = 
\begin{bmatrix}
0 & 1\\
0 & 0\\
\end{bmatrix}
\xi
+ 
\begin{bmatrix}
0\\
1\\
\end{bmatrix}
\upsilon,\right.
\end{equation}
with the set of states restricted to $[0,3.2]\times[-1.5,1.5]$, state quantization parameters as $(0.2, 0.3)$, input set restricted to $[-0.3,0.3]$, input quantization parameter of $0.2$, and sampling time $\tau=0.3$. 
The third case study, denoted by \textsf{Robot}, correspond to a mobile robot whose dynamics is given by \cite{belta}:
\begin{equation*}
\Sigma\left\{\dot\xi 
= 
\begin{bmatrix}
\upsilon_1\\
\upsilon_1\\
\end{bmatrix}.\right.
\end{equation*}
The states represent the position of the robot. We consider the state set restricted to $[0,63]\times[0,63]$ and state quantization parameter as $1$. The input set is restricted to $[-1,1]\times[-1,1]$ with input quantization parameter of $1$, and sampling time is $\tau = 1$.
The last three case studies, denoted by \textsf{Vehicle1}, \textsf{Vehicle2}, \textsf{Vehicle3}, respectively, correspond to a vehicle whose dynamics is given by:
\begin{equation}
\label{EQN_VEHICLE}
\Sigma\left\{\dot\xi 
= 
\begin{bmatrix}
\upsilon_1 \cos(\alpha + \xi_3) \cos(\alpha)^{-1} \\
\upsilon_1 \sin(\alpha + \xi_3) \cos(\alpha)^{-1}\\
\upsilon_1 \tan(\upsilon_2)\\
\end{bmatrix},\right.
\end{equation}
where $\alpha = \arctan(\tan(\upsilon_2)/2)$.
The first and second states represent the position of the vehicle while the third represents the heading angle. The control inputs represent rear wheel velocity and the steering angle. We consider state quantization parameter as $0.2$, input set restricted to $[-1,1]\times[-1,1]$, and input set quantization parameter as $0.3$, and a sampling time of $\tau = 0.3$ for the last three case studies.
We consider the state set restricted to $[0,6]\times[0,5]\times[-3.54,3.54]$ for \textsf{Vehicle1}, \textsf{Vehicle2} case studies. The state set is restricted to $[0,10]\times[0,10]\times[-3.54,3.54]$ for the \textsf{Vehicle3} case study. 
Some parts of the state sets of the last four case studies were removed to represent obstacles that need to be avoided when synthesizing the symbolic controllers.
The symbolic models were constructed using a PC (Intel Core i7 3.6 GHz and 32 GB RAM). The CUDD library \cite{cudd} was used to operate with BDDs. Note that the inconsistencies in the execution time and storage memory reported in Table \ref{TBL_MODEL_CONS_DETAILS} are due to the heuristic algorithms implemented in the CUDD library for operating with BDDs to automatically reorder binary variables for optimizing BDD operations. 
We also implemented the construction of symbolic models of NCS using the schemes proposed in \cite{pola4,pola5}. The computation time and memory storage for the construction of symbolic model for NCS containing \textsf{DI} with delay parameters $N_{\max}^{\text{sc}}=N_{\max}^{\text{ca}}=2$ amounted to $1.17$ seconds and $42.6$ KB respectively. For the \textsf{Vehicle1} case with delay parameters $N_{\max}^{\text{sc}}=N_{\max}^{\text{ca}}=2$, the computation time amounted to more than two days and the memory usage exceeded 32 GB. This shows that the computation times and memory required to construct symbolic models using the schemes in \cite{pola4,pola5} are several orders of magnitude more than those using the proposed scheme in this paper which amounted to $0.019$ and $117.4$ seconds, respectively (including the computation time required by the tool \texttt{SCOTS} to construct the symbolic models of the plants inside the NCS), while the storage memory is already reported in table \ref{TBL_MODEL_CONS_DETAILS}.
\subsection{Controller synthesis and refinement: the Robot case}
We consider the third case study from Table \ref{TBL_MODEL_CONS_DETAILS} with the network delays $(N_{\max}^{\text{sc}} = 2, N_{\max}^{\text{ca}} = 2)$. The control objective is to enforce the robot to infinitely-often visit two target sets of states described by propositions $\textsf{Target}1$ and $\textsf{Target}1$ which are defined by the hyper-intervals $[5,15]\times[45,55]$ and $[45,55]\times[5,15]$, respectively. Moreover, the robot needs to avoid a set of nine obstacles defined by the propositions $\text{Obstacle}_i$, $i\in\{1,\ldots,9\},$ which are defined by the hyper-intervals $[5,15]\times[20, 22]$, $[15,17]\times[5,22]$, $[48,50]\times[45,60]$, $[51,58]\times[45,47]$, $[27,36]\times[20,45]$, $[44,49]\times[27,36]$, $[27,36]\times[52,57]$, $[27,36]\times[5,10]$, and $[14,19]\times[27,36]$, respectively. This control objective can be described by the following LTL formula:
\begin{equation}
\nonumber
\psi = \Big( \underset{i=1}{\overset{9}{\bigwedge}} \square(\lnot \textsf{Obstacle}_i)\Big) \wedge \square\Diamond(\textsf{Target}1) \wedge \square\Diamond(\textsf{Target}2). 
\end{equation}
The controller was synthesized using fixed-point computations as implemented in \texttt{SENSE}. Remark that the resulting controller is a dynamic controller with two discrete states. 
The computation of the symbolic controller amounted to $4.7$ seconds. Figure \ref{FIG_EX_SIM_MATLAB_ROBOT} shows the closed loop simulation of the NCS. 
\begin{figure}
	\centering
	\includegraphics[width=0.55\textwidth]{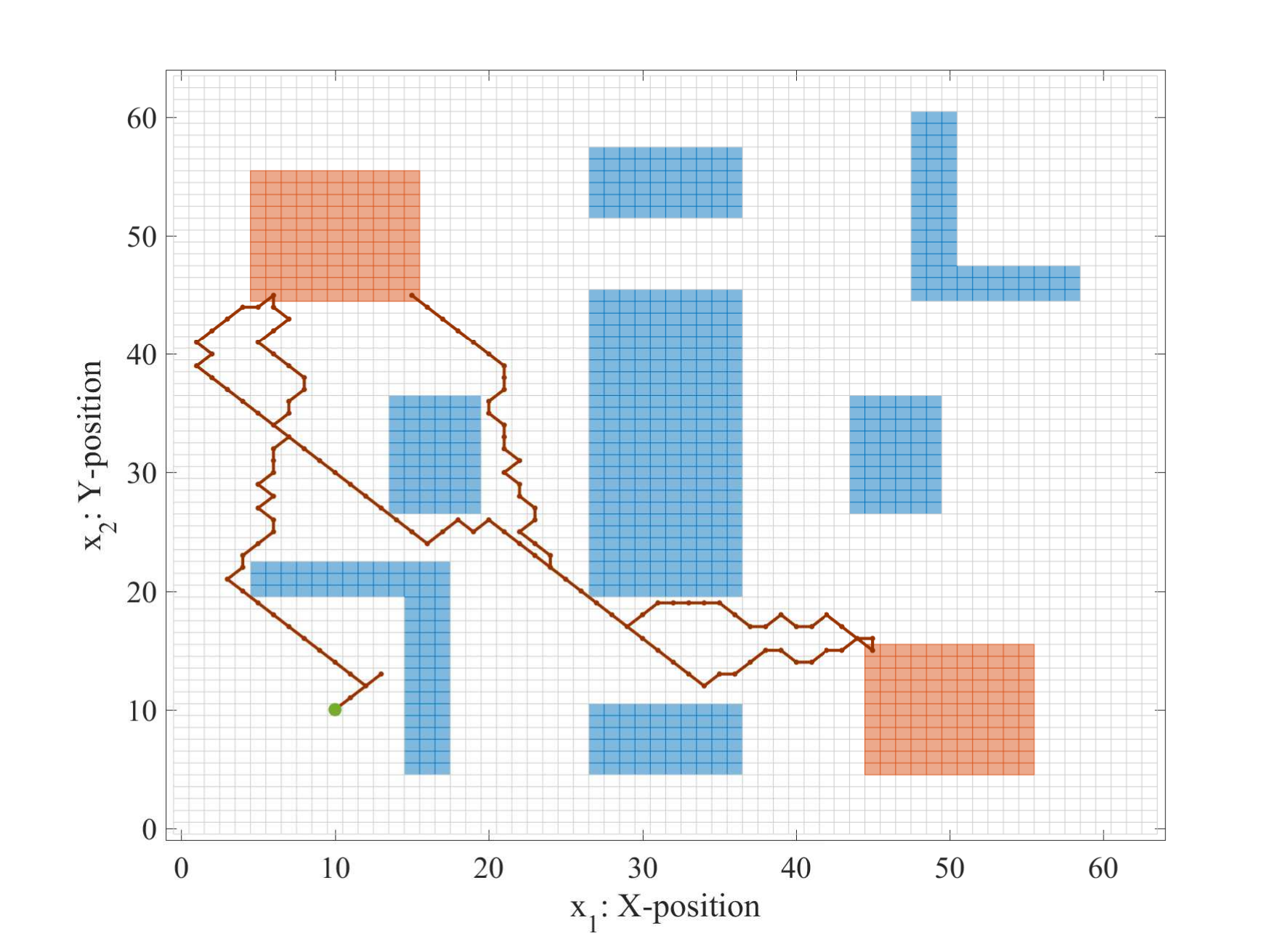}
	\caption{Closed loop simulation of the NCS with the robot system in \texttt{MATLAB}. The target sets are indicated with the red boxes and obstacles with the blue boxes.}
	\label{FIG_EX_SIM_MATLAB_ROBOT}
\end{figure}
For a more realistic simulation environment, we consider \texttt{OMNeT++} \cite{omnet++}, a common simulation framework for networks. Communication channels are modeled using a random propagation-delay communication channels in \texttt{OMNeT++}. Figure \ref{FIG_EX_SIM_OMNET_ROBOT} shows the closed loop simulation results in \texttt{OMNeT++}. We make use of the animation capabilities of \texttt{OMNeT++} to visualize both packet transfers over the network as well as the movement of the robot through the state set as illustrated in \cite{sense}. Controller synthesis and refinement for the vehicle dynamic in \eqref{EQN_VEHICLE}, for a configuration of network delays, and for an LTL specification are provided in \cite{sense}.
\begin{figure}
	\centering
	\includegraphics[width=0.55\textwidth]{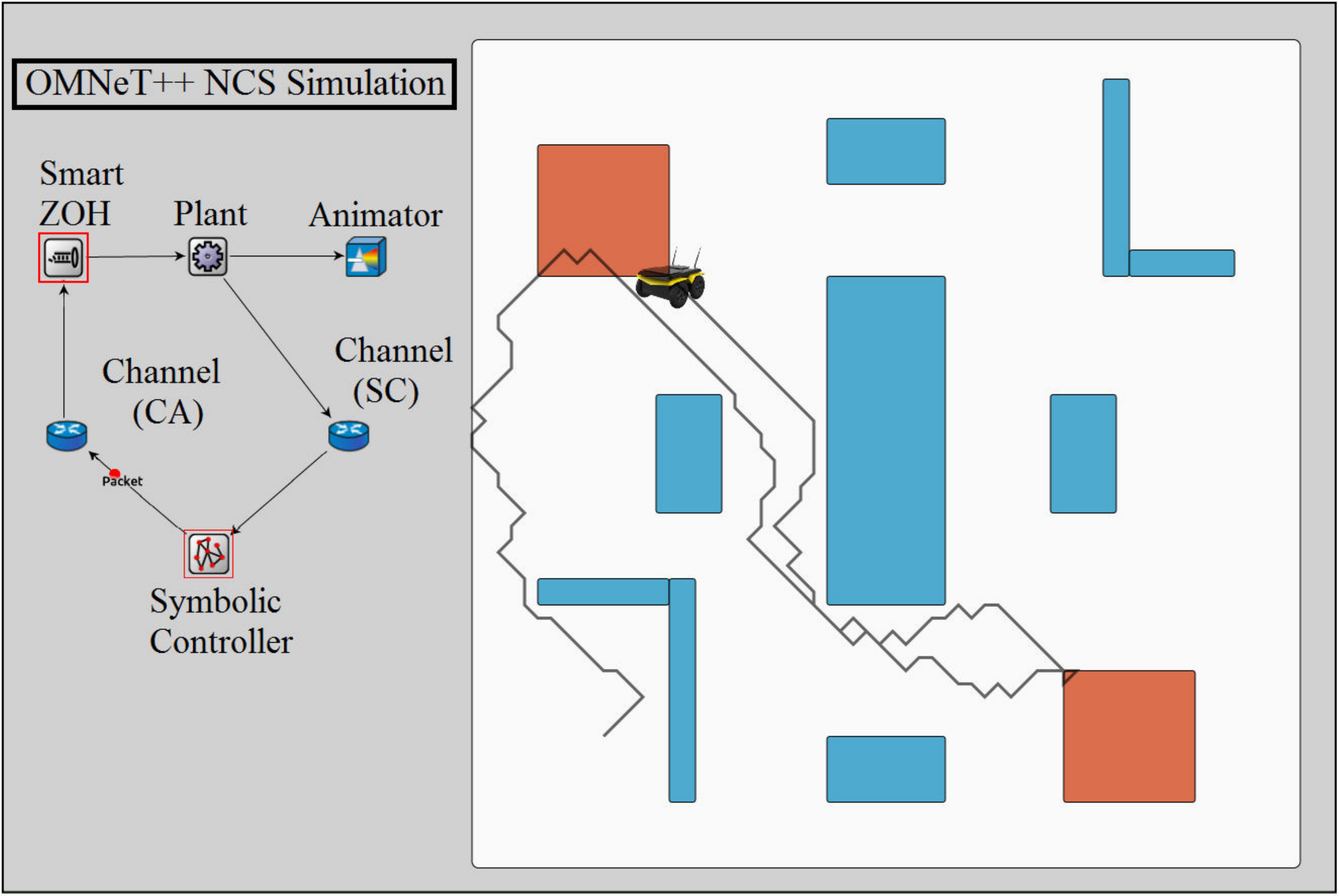}
	\caption{Closed loop simulation of the NCS with the robot system in \texttt{OMNeT++}. On the left, we illustrate how the packets move between different parts of the network. On the right, the movement of the robot over the state set is illustrated.}
	\label{FIG_EX_SIM_OMNET_ROBOT}
\end{figure}

\section{Relationship to Adjacent Work}
Both our work and the ones in \cite{pola4,pola5} explicitly consider the network non-idealities (i), (ii), and (iv) acting on the NCS simultaneously. The results in \cite{pola4,pola5} provide symbolic models obtained via gridding techniques (discretization of state and control sets), which practically are likely to severely suffer from the curse of dimensionality. However, in our proposed framework one can directly employ available and well investigated symbolic models obtained exclusively for the plant, including full grid-based approaches \cite{girard2,majid}, or newer partial grid-based ones \cite{majid10}, and non-grid-based ones \cite{boyan}, and then construct symbolic models for the overall NCS. While the results in \cite{pola4,pola5} do not consider the possibility of out-of-order packet arrivals and message rejections, i.e. the effect of older data being neglected because more recent data is available, we consider them in this work. The results in \cite{pola4,pola5} only consider static (i.e. memoryless) symbolic controllers, whereas general temporal logic specifications often are shown to require dynamic (i.e. with memory) symbolic controllers \cite{katoen08}, which are indeed allowed in our framework (cf. the example section). While the results in \cite{pola4,pola5} can only address specifications expressed in terms of specific types of nondeterministic automata, our results enable the study of larger classes of logical specifications, 
such as those expressed as general LTL formulae (cf. the example section) or as automata on infinite strings.

Besides these differences, the fundamental distinguishing feature of our work with respect to the recent contributions in \cite{pola4,pola5} is the nature of the triggering mechanism for message transmission: 
\cite{pola4,pola5} consider an event-triggered mechanism, 
in which new sensor measurements are transmitted only once the actuator is updated with the control input computed using the last transmitted measurement. 
While preventing measurements from arriving out of order, this restricts the applicability to systems in which sensors and actuators are co-located. In our approach, on the other hand, the sensors and controllers send new measurements/control updates in a periodic fashion.  This forces dealing explicitly with out-of-order messages, but in exchange it removes any restriction on the location of sensors, controllers, or actuators. Note that additionally, our formulation still allows to capture implementations with transmissions of measurements/control updates triggered by the satisfaction of certain conditions (i.e. event-based control), by encoding such restrictions in the plant model.

\section{Discussion and Conclusions}
In this paper we have provided a construction of symbolic models for NCS, subject to the following non-idealities: 
variable communication delays, quantization errors, packet losses, and limited bandwidth. 
This novel approach is practically relevant since it can leverage any existing symbolic model for the plant, 
and in particular is not limited to grid-based ones and extendible to work over stochastic plants -- both features are current focus of active investigation elsewhere. 
Furthermore, this approach can be applied to treat any specification expressed as a formula in LTL (cf. the example section) or as an automaton on infinite strings, without requiring any additional re-formulation.  

Future work will concentrate on the following goals: 
1) the construction of symbolic models for NCS with explicit probabilistic structure over the transmission intervals, communication delays, and packet dropouts; 
2) the construction of symbolic models for still more general NCS,  
by considering additional network non-idealities, in particular time-varying sampling and transmission intervals; and
3) the study of interconnections and synthesis employing the different outputs enabled by our abstractions at the sensor and controller side.

\section{Acknowledgments}
The authors would like to thank Matthias Rungger for fruitful technical discussions over Subsection \ref{band}.

\bibliographystyle{alpha}
\bibliography{reference}

\newcommand{\etalchar}[1]{$^{#1}$}
\begin{thebibliography}{CvdWHN09}

\bibitem[ADJ{\etalchar{+}}11]{alur2}
R.~Alur, A.~D'Innocenzo, K.~H. Johansson, G.~J. Pappas, and G.~Weiss.
\newblock Compositional modeling and analysis of multi-hop control networks.
\newblock {\em IEEE Transactions on Automatic Control}, 56(10):2345--2357,
  2011.

\bibitem[AHS12]{antunes}
D.~Antunes, J.~P. Hespanha, and C.~Silvestre.
\newblock Volterra integral approach to impulsive renewal systems:
  {A}pplication to networked control.
\newblock {\em IEEE Transactions on Automatic Control}, 57(3):607--619, March
  2012.

\bibitem[Ang02]{angeli}
D.~Angeli.
\newblock A {L}yapunov approach to incremental stability properties.
\newblock {\em IEEE Transactions on Automatic Control}, 47(3):410--421, 2002.

\bibitem[AS99]{sontag}
D.~Angeli and E.~D. Sontag.
\newblock Forward completeness, unboundedness observability, and their
  {L}yapunov characterizations.
\newblock {\em Systems and Control Letters}, 38:209--217, 1999.

\bibitem[BK04]{belta}
C.~Belta and V.~Kumar.
\newblock Abstractions and control for groups of robots.
\newblock {\em IEEE Transactions on Robotics}, 20(5):865--875, October 2004.

\bibitem[BK08]{katoen08}
C.~Baier and J.~P. Katoen.
\newblock {\em Principles of model checking}.
\newblock The MIT Press, April 2008.

\bibitem[BMH12]{bauer}
N.~W. Bauer, P.~J.~H. Maas, and W.~P. M.~H. Heemels.
\newblock Stability analysis of networked control systems: a sum of squares
  approach.
\newblock {\em Automatica}, 48(8):1514--1524, 2012.

\bibitem[BPD12a]{pola5}
A.~Borri, G.~Pola, and {M.D.} {Di Benedetto}.
\newblock Integrated symbolic design of unstable nonlinear networked control
  systems.
\newblock {\em in Proceedings of 51th IEEE Conference on Decision and Control},
  December 2012.

\bibitem[BPD12b]{pola4}
A.~Borri, G.~Pola, and {M.D.} {Di Benedetto}.
\newblock A symbolic approach to the design of nonlinear networked control
  systems.
\newblock {\em in Proceedings of 15th International Conference on Hybrid
  Systems: Computation and Control}, pages 255--264, April 2012.

\bibitem[CDHR06]{chatterjee}
K.~Chatterjee, L.~Doyen, T.~A. Henzinger, and J.~F. Raskin.
\newblock Algorithms for omega-regular games with imperfect information.
\newblock In Z.~Esik, editor, {\em Computer Science Logic}, volume 4207 of {\em
  Lecture Notes in Computer Science}, pages 287--302. Springer Berlin
  Heidelberg, 2006.

\bibitem[CvdWHN09]{marieke}
M.~B.~G. Cloosterman, N.~van~de Wouw, W.~P. M.~H. Heemels, and H.~Nijmeijer.
\newblock Stability of networked control systems with uncertain time-varying
  delays.
\newblock {\em IEEE Transactions on Automatic Control}, 54(7):1575--1580, July
  2009.

\bibitem[DH04]{dai2004tsync}
H.~Dai and R.~Han.
\newblock Tsync: a lightweight bidirectional time synchronization service for
  wireless sensor networks.
\newblock {\em ACM SIGMOBILE Mobile Computing and Communications Review},
  8(1):125--139, 2004.

\bibitem[EGE02]{elson2002fine}
J.~Elson, L.~Girod, and D.~Estrin.
\newblock Fine-grained network time synchronization using reference broadcasts.
\newblock {\em ACM SIGOPS Operating Systems Review}, 36(SI):147--163, 2002.

\bibitem[ER16]{slugs}
R.~Ehlers and V.~Raman.
\newblock Slugs: {E}xtensible {GR(1)} synthesis.
\newblock In S.~Chaudhuri and A.~Farzan, editors, {\em Computer Aided
  Verification (CAV)}, volume 9780 of {\em Lecture Notes in Computer Science},
  pages 333--339. Springer International Publishing, July 2016.

\bibitem[GCL08]{gao}
H.~Gao, T.~Chen, and J.~Lam.
\newblock A new delay system approach to network-based control.
\newblock {\em Automatica}, 44(1):39--52, 2008.

\bibitem[Gir10]{girard4}
A.~Girard.
\newblock Synthesis using approximately bisimilar abstractions: state-feedback
  controllers for safety specifications.
\newblock {\em in Proceedings of 13th International Conference on Hybrid
  Systems: Computation and Control}, pages 111--120, April 2010.

\bibitem[Gir13]{girard6}
A.~Girard.
\newblock Low-complexity quantized switching controllers using approximate
  bisimulation.
\newblock {\em Nonlinear Analysis: Hybrid Systems}, 10:34--44, November 2013.

\bibitem[GP07]{girard}
A.~Girard and G.~J. Pappas.
\newblock Approximation metrics for discrete and continuous systems.
\newblock {\em IEEE Transactions on Automatic Control}, 25(5):782--798, 2007.

\bibitem[GPT09]{girard2}
A.~Girard, G.~Pola, and P.~Tabuada.
\newblock Approximately bisimilar symbolic models for incrementally stable
  switched systems.
\newblock {\em IEEE Transactions on Automatic Control}, 55(1):116--126, 2009.

\bibitem[HvdW10]{heemels}
W.~P. M.~H. Heemels and N.~van~de Wouw.
\newblock Stability and stabilization of networked control systems.
\newblock In A.~Bemporad, W.~P. M.~H. Heemels, and M.~Johansson, editors, {\em
  Networked Control Systems}, volume 406 of {\em Lecture Notes in Control and
  Information Sciences}, pages 203--253. Springer London, 2010.

\bibitem[KRZar]{KRZ}
M.~Khaled, M.~Rungger, and M.~Zamani.
\newblock Symbolic models of networked control systems: A feedback refinement
  relation approach.
\newblock In {\em 54th Annual Allerton Conference on Communication, Control,
  and Computing}, 2016, to appear.

\bibitem[KZ16]{sense}
M.~Khaled and M.~Zamani.
\newblock http://www.hcs.ei.tum.de/software/sense, November 2016.

\bibitem[MPS95]{MalerPnueliSifakis95}
O.~Maler, A.~Pnueli, and J.~Sifakis.
\newblock On the synthesis of discrete controllers for timed systems.
\newblock {\em in Proceedings of 12th Annual Symposium on Theoretical Aspects
  of Computer Science}, 900:229--242, 1995.

\bibitem[NL09]{nesic}
D.~Nesic and D.~Liberzon.
\newblock A unified framework for design and analysis of networked and
  quantized control systems.
\newblock {\em IEEE Transactions on Automatic Control}, 54(4):732--747, 2009.

\bibitem[PT09]{pola1}
G.~Pola and P.~Tabuada.
\newblock Symbolic models for nonlinear control systems: alternating
  approximate bisimulations.
\newblock {\em SIAM Journal on Control and Optimization}, 48(2):719--733, 2009.

\bibitem[RWR16]{Matthias_Gunther}
G.~Reissig, A.~Weber, and M.~Rungger.
\newblock Feedback refinement relations for the synthesis of symbolic
  controllers.
\newblock {\em IEEE Transactions on Automatic Control (doi:
  10.1109/TAC.2016.2593947)}, 2016.

\bibitem[RZar]{scots}
M.~Rungger and M.~Zamani.
\newblock {SCOTS}: A tool for the synthesis of symbolic controllers.
\newblock In {\em Proceedings of the 19th International Conference on Hybrid
  Systems: Computation and Control}. ACM New York, NY, USA, April 2016, to
  appear.

\bibitem[Som15]{cudd}
F.~Somenzi.
\newblock {\em {CUDD}: {CU} Decision Diagram Package}, 3.0.0 edition, December
  2015.

\bibitem[Son98]{sontag1}
E.~D. Sontag.
\newblock {\em Mathematical control theory: Deterministic finite dimensional
  systems}, volume~6.
\newblock Springer-Verlag, New York, 2nd edition, 1998.

\bibitem[Tab09]{paulo}
P.~Tabuada.
\newblock {\em Verification and Control of Hybrid Systems, A symbolic
  approach.}
\newblock Springer US, 2009.

\bibitem[vdWNH12]{nathan}
N.~van~de Wouw, D.~Nesic, and W.~P. M.~H. Heemels.
\newblock A discrete-time framework for stability analysis of nonlinear
  networked control systems.
\newblock {\em Automatica}, 48(6):1144--1153, June 2012.

\bibitem[VH08]{omnet++}
A.~Varga and R.~Hornig.
\newblock An overview of the {OMNeT++} simulation environment.
\newblock In {\em Proceedings of the 1st international conference on Simulation
  tools and techniques for communications, networks and systems \& workshops},
  pages 1--10, 2008.

\bibitem[YTC{\etalchar{+}}13]{boyan}
B.~Yordanov, J.~Tumova, I.~Cerna, J.~Barnat, and C.~Belta.
\newblock Formal analysis of piecewise affine systems through fomula-guided
  refinement.
\newblock {\em Automatica}, 49(1):261--266, January 2013.

\bibitem[ZEAL13]{majid7}
M.~Zamani, P.~Mohajerin Esfahani, A.~Abate, and J.~Lygeros.
\newblock Symbolic models for stochastic control systems without stability
  assumptions.
\newblock In {\em Proceedings of European Control Conference}, pages
  4257--4262, July 2013.

\bibitem[ZEM{\etalchar{+}}14]{majid8}
M.~Zamani, P.~Mohajerin Esfahani, R.~Majumdar, A.~Abate, and J.~Lygeros.
\newblock Symbolic control of stochastic systems via approximately bisimilar
  finite abstractions.
\newblock {\em IEEE Transactions on Automatic Control, Special Issue on Control
  of Cyber-Physical Systems}, 59(12):3135--3150, December 2014.

\bibitem[ZMA14]{majid13}
M.~Zamani, M.~{Mazo Jr.}, and A.~Abate.
\newblock Finite abstractions of networked control systems.
\newblock In {\em Proceedings of the 53rd IEEE Conference on Decision and
  Control}, pages 95--100, December 2014.

\bibitem[ZPMT12]{majid}
M.~Zamani, G.~Pola, M.~{Mazo Jr.}, and P.~Tabuada.
\newblock Symbolic models for nonlinear control systems without stability
  assumptions.
\newblock {\em IEEE Transactions on Automatic Control}, 57(7):1804--1809, July
  2012.

\bibitem[ZTA14]{majid10}
M.~Zamani, I.~Tkachev, and A.~Abate.
\newblock Bisimilar symbolic models for stochastic control systems without
  state-space discretization.
\newblock In {\em Proceedings of the 17th International Conference on Hybrid
  Systems: Computation and Control}, pages 41--50. ACM New York, NY, USA, April
  2014.

\bibitem[ZTAng]{ZTA1}
M.~Zamani, I.~Tkachev, and A.~Abate.
\newblock Towards scalable synthesis of stochastic control systems.
\newblock {\em Discrete Event Dynamic Systems}, 2016, forthcoming.

\end{thebibliography}

\section{Appendix}
\begin{proof}[Proof of Theorem \ref{main_theorem}]
We start by proving $S_*(\widetilde{\Sigma})\preceq_{\mathcal{AS}}^{\varepsilon}S(\widetilde{\Sigma})$. Since $S_\params(\Sigma)\preceq_{\mathcal{AS}}^{\varepsilon}S_\tau(\Sigma)$, there exists an alternating $\varepsilon$-approximate simulation relation ${R}$ from $S_\params(\Sigma)$ to $S_\tau(\Sigma)$. Consider the relation $\widetilde{R}\subseteq X_{*}\times X$ defined by $\left(\mathsf{x}_*,\mathsf{x}\right)\in{\widetilde{R}}$, where $\mathsf{x}_*=\Big(x_{*1},\ldots,x_{*N^{\text{sc}}_{\max}},u_{*1},\ldots,u_{*N^{\text{ca}}_{\max}},\widetilde{N}_{*1},\ldots,\widetilde{N}_{*N_{\max}^{\text{sc}}},\widehat{N}_{*1},\ldots,\widehat{N}_{*N^{\text{ca}}_{\max}}\Big)$ and $\mathsf{x}=\Big(x_1,\ldots,x_{N^{\text{sc}}_{\max}},\upsilon_{1},\ldots,\upsilon_{N^{\text{ca}}_{\max}},\widetilde{N}_{1},\ldots,\widetilde{N}_{N_{\max}^{\text{sc}}},\widehat{N}_{1},\ldots,\widehat{N}_{N^{\text{ca}}_{\max}}\Big)$, if and only if $\widetilde{N}_{*i}=\widetilde{N}_{i}$, $\forall i\in\left[1;N^{\text{sc}}_{\max}\right]$, $\widehat{N}_{*j}=\widehat{N}_{j}$, $\forall j\in[1;N^{\text{ca}}_{\max}]$, $\left(x_{*k},x_{k}\right)\in{R}$, $\forall k\in\left[1;N^{\text{sc}}_{\max}\right]$, and for each $u_{*i}$ and the corresponding $\upsilon_i$ there exists $x'_*\in\textbf{Post}_{u_{*i}}(x_*)$ such that $\left(x'_*,\xi_{x\upsilon_i}(\tau)\right)\in{R}$ for any $i\in\left[1;N_{\max}^{\text{ca}}\right]$ and any $\left(x_*,x\right)\in{R}$. Note that if $U_\tau=U_\params$ and they are finite then the last condition of the relation $\widetilde{R}$ is nothing more than requiring $u_{*i}=\upsilon_i$ for any $i\in\left[1;N_{\max}^{\text{ca}}\right]$.

Consider $\mathsf{x}_{*0}\Let\big(x_{*0},q,\ldots,q,u_{*0},\ldots,u_{*0},N^{\text{sc}}_{\max},\ldots,N^{\text{sc}}_{\max},N^{\text{ca}}_{\max},\ldots,N^{\text{ca}}_{\max}\big)\in X_{*0}$. Due to the relation $R$, there exist $x_0\in X_{\tau0}$ such that $\left(x_{*0},x_{0}\right)\in{R}$ and $\upsilon_0\in U_\tau$ such that there exists $x'_*\in\mathbf{Post}_{u_{*0}}\left(x_{*}\right)$ satisfying $\left(x'_*,\xi_{x\upsilon_0}(\tau)\right)\in{R}$ for any $\left(x_*,x\right)\in{R}$. Hence, by choosing $\mathsf{x}_0\Let(x_{0},q,\ldots,q,\upsilon_{0},\ldots,\upsilon_{0},N^{\text{sc}}_{\max},\ldots,N^{\text{sc}}_{\max},N^{\text{ca}}_{\max},\\\ldots,N^{\text{ca}}_{\max})\in X_{0}$, one gets $\left(\mathsf{x}_{*0},\mathsf{x}_0\right)\in \widetilde{R}$ and condition (i) in Definition \ref{AASR} is satisfied.

Now consider any $\left(\mathsf{x}_*,\mathsf{x}\right)\in{\widetilde{R}}$, where $\mathsf{x}_*=\Big(x_{*1},\ldots,x_{*N^{\text{sc}}_{\max}},u_{*1},\ldots,u_{*N^{\text{ca}}_{\max}},\widetilde{N}_{*1},\ldots,\widetilde{N}_{*N_{\max}^{\text{sc}}},\widehat{N}_{*1},\ldots,\widehat{N}_{*N^{\text{ca}}_{\max}}\Big)$ and $\mathsf{x}=\Big(x_1,\ldots,x_{N^{\text{sc}}_{\max}},\upsilon_{1},\ldots,\upsilon_{N^{\text{ca}}_{\max}},\widetilde{N}_{1},\ldots,\widetilde{N}_{N_{\max}^{\text{sc}}},\widehat{N}_{1},\ldots,\widehat{N}_{N^{\text{ca}}_{\max}}\Big)$. Using definitions of $S_*(\widetilde{\Sigma})$ and $S(\widetilde{\Sigma})$, one obtains $H_*\left(\mathsf{x}_*\right)=H_\params(x_{*1})$ and $H\left(\mathsf{x}\right)=H_\tau(x_{1})$. Since $\left(x_{*1},x_{1}\right)\in{R}$, one gets $\mathsf{d}_{Y_\tau}\left(H_\params\left(x_{*1}\right),H_\tau\left(x_1\right)\right)\leq\varepsilon$ and, hence, condition (ii) in Definition \ref{AASR} is satisfied.

Let us now show that condition (iii) in Definition \ref{AASR} holds. Consider any $\left(\mathsf{x}_*,\mathsf{x}\right)\in{\widetilde{R}}$, where $\mathsf{x}_*=\Big(x_{*1},\ldots,\\x_{*N^{\text{sc}}_{\max}},u_{*1},\ldots,u_{*N^{\text{ca}}_{\max}},\widetilde{N}_{*1},\ldots,\widetilde{N}_{*N_{\max}^{\text{sc}}},\widehat{N}_{*1},\ldots,\widehat{N}_{*N^{\text{ca}}_{\max}}\Big)$, $\mathsf{x}=\Big(x_1,\ldots,x_{N^{\text{sc}}_{\max}},\upsilon_{1},\ldots,\upsilon_{N^{\text{ca}}_{\max}},\widetilde{N}_{1},\ldots,\widetilde{N}_{N_{\max}^{\text{sc}}},\\\widehat{N}_{1},\ldots,\widehat{N}_{N^{\text{ca}}_{\max}}\Big)$. 
Consider any $u_*\in{U_*(\mathsf{x}_*)}$. Using the relation ${R}$, there exist $\upsilon\in U(\mathsf{x})$ and $\grave{x}_*\in\mathbf{Post}_{u_*}(x_{*})$ such that $(\grave{x}_*,\xi_{x\upsilon}(\tau))\in{R}$ for any $(x_*,x)\in{R}$. Now consider any $\mathsf{x}'=\Big(x',x_1,\ldots,x_{N^{\text{sc}}_{\max}-1},\upsilon,\upsilon_{1},\ldots,\upsilon_{N^{\text{ca}}_{\max}-1},\\\widetilde{N},\widetilde{N}_{1},\ldots,\widetilde{N}_{N_{\max}^{\text{sc}}-1},\widehat{N},\widehat{N}_{1},\ldots,\widehat{N}_{N^{\text{ca}}_{\max}-1}\Big)\in\mathbf{Post}_{\upsilon}(\mathsf{x})\subseteq{X}$ for some $\widetilde{N}\in\left[N^{\text{sc}}_{\min};N^{\text{sc}}_{\max}\right]$ and $\widehat{N}\in\left[N^{\text{ca}}_{\min};N^{\text{ca}}_{\max}\right]$ where $x'=\xi_{x_1\upsilon_k}\left(\tau\right)$ for some given $k\in\left[N^{\text{ca}}_{\min};N^{\text{ca}}_{\max}\right]$ (cf. Definition $S(\widetilde{\Sigma})$). Because of the relation $R$, there exists $x'_*\in\mathbf{Post}_{u_{*k}}(x_{*1})$ in $S_\params(\Sigma)$ such that $\left(x'_*,x'\right)\in{R}$. Hence, due to the definition $S_*(\widetilde\Sigma)$, one can choose $\mathsf{x}'_*=\Big(x'_*,x_{*1},\ldots,x_{*\left(N^{\text{sc}}_{\max}-1\right)},u_*,u_{*1},\ldots,u_{*\left(N^{\text{ca}}_{\max}-1\right)},\widetilde{N},\widetilde{N}_{1},\ldots,\widetilde{N}_{N_{\max}^{\text{sc}}-1},\widehat{N},\widehat{N}_{1},\ldots,\widehat{N}_{N^{\text{ca}}_{\max}-1}\Big)\in\mathbf{Post}_{u_*}(\mathsf{x}_*)\subseteq{X}_*$. Due to the relation ${R}$, one can readily verify that $\mathsf{d}_{Y_\tau}\left( H_\params(x'_*),H_\tau(x')\right)\leq\varepsilon$. Hence, one gets $\mathsf{d}_{Y}\left(H_*\left(\mathsf{x}'_*\right),H\left(\mathsf{x}'\right)\right)=\mathsf{d}_{Y_\tau}\left(H_\params\left(x'_{*}\right),H_\tau\left(x'\right)\right)\leq\varepsilon$. Hence, $\left(\mathsf{x}'_*,\mathsf{x}'\right)\in{\widetilde{R}}$ implying that condition (iii) in Definition \ref{AASR} holds.

Now we prove $S(\widetilde\Sigma)\preceq_{\mathcal{S}}^{\varepsilon}S_*(\widetilde\Sigma)$. Since $S_\tau(\Sigma)\preceq_{\mathcal{S}}^{\varepsilon}S_\params(\Sigma)$, there exists an $\varepsilon$-approximate simulation relation ${R}$ from $S_\tau(\Sigma)$ to $S_\params(\Sigma)$. Consider the relation $\widetilde{R}\subseteq X\times X_*$ defined by $\left(\mathsf{x},\mathsf{x}_*\right)\in{\widetilde{R}}$, where $\mathsf{x}=\Big(x_1,\ldots,x_{N^{\text{sc}}_{\max}},\upsilon_{1},\ldots,\upsilon_{N^{\text{ca}}_{\max}},\widetilde{N}_{1},\ldots,\widetilde{N}_{N_{\max}^{\text{sc}}},\widehat{N}_{1},\ldots,\widehat{N}_{N^{\text{ca}}_{\max}}\Big)$ and $\mathsf{x}_*=\Big(x_{*1},\ldots,x_{*N^{\text{sc}}_{\max}},u_{*1},\ldots,u_{*N^{\text{ca}}_{\max}},\widetilde{N}_{*1},\\\ldots,\widetilde{N}_{*N_{\max}^{\text{sc}}},\widehat{N}_{*1},\ldots,\widehat{N}_{*N^{\text{ca}}_{\max}}\Big)$, if and only if $\widetilde{N}_{i}=\widetilde{N}_{*i}$, $\forall i\in\left[1;N^{\text{sc}}_{\max}\right]$, $\widehat{N}_{j}=\widehat{N}_{*j}$, $\forall j\in[1;N^{\text{ca}}_{\max}]$, $\left(x_{k},x_{*k}\right)\in{R}$, $\forall k\in\left[1;N^{\text{sc}}_{\max}\right]$, and for each $\upsilon_i$ and the corresponding $u_{*i}$ there exists a $x'_*\in\textbf{Post}_{u_{*i}}(x_*)$ such that $\left(\xi_{x\upsilon_i}(\tau),x'_*\right)\in{R}$ for any $i\in\left[1;N_{\max}^{\text{ca}}\right]$ and any $\left(x,x_*\right)\in{R}$. Note that if $U_\tau=U_\params$ and they are finite then the last condition of the relation $\widetilde{R}$ is nothing more than requiring $u_{*i}=\upsilon_i$ for any $i\in\left[1;N_{\max}^{\text{ca}}\right]$.

Consider $\mathsf{x}_{0}\Let\big(x_{0},q,\ldots,q,\upsilon_{0},\ldots,\upsilon_{0},N^{\text{sc}}_{\max},\ldots,N^{\text{sc}}_{\max},N^{\text{ca}}_{\max},\ldots,N^{\text{ca}}_{\max}\big) \in X_{0}$. Due to the relation ${R}$, there exist $x_{*0}\in X_{*0}$ such that $\left(x_{0},x_{*0}\right)\in{R}$ and $u_{*0}\in U_\params$ such that there exists $x'_*\in\mathbf{Post}_{u_{*0}}\left(x_{*}\right)$ satisfying $(\xi_{x\upsilon_0}(\tau),x'_*)\in{R}$ for any $(x,x_*)\in{R}$. Hence, by choosing $\mathsf{x}_{*0}\Let\big(x_{*0},q,\ldots,q,u_{*0},\ldots,u_{*0},N^{\text{sc}}_{\max},\ldots,N^{\text{sc}}_{\max},\\N^{\text{ca}}_{\max},\ldots,N^{\text{ca}}_{\max}\big)\in X_{*0}$, one gets $\left(\mathsf{x}_{0},\mathsf{x}_{*0}\right)\in \widetilde{R}$ and condition (i) in Definition \ref{ASR} is satisfied.

Now consider any $\left(\mathsf{x},\mathsf{x}_*\right)\in{\widetilde{R}}$, where $\mathsf{x}=\Big(x_1,\ldots,x_{N^{\text{sc}}_{\max}},\upsilon_{1},\ldots,\upsilon_{N^{\text{ca}}_{\max}},\widetilde{N}_{1},\ldots,\widetilde{N}_{N_{\max}^{\text{sc}}},\widehat{N}_{1},\ldots,\widehat{N}_{N^{\text{ca}}_{\max}}\Big)$ and $\mathsf{x}_*=\Big(x_{*1},\ldots,x_{*N^{\text{sc}}_{\max}},u_{*1},\ldots,u_{*N^{\text{ca}}_{\max}},\widetilde{N}_{*1},\ldots,\widetilde{N}_{*N_{\max}^{\text{sc}}},\widehat{N}_{*1},\ldots,\widehat{N}_{*N^{\text{ca}}_{\max}}\Big)$. Using definitions of $S(\widetilde{\Sigma})$ and $S_*(\widetilde{\Sigma})$, one obtains $H\left(\mathsf{x}\right)=H_\tau(x_{1})$ and $H_*\left(\mathsf{x}\right)=H_\params(x_{*1})$. Since $\left(x_{1},x_{*1}\right)\in{R}$, one gets $\mathsf{d}_{Y_\tau}\left(H_\tau\left(x_{1}\right),H_\params\left(x_{*1}\right)\right)\\\leq\varepsilon$ and, hence, condition (ii) in Definition \ref{ASR} is satisfied.

Let us now show that condition (iii) in Definition \ref{ASR} holds. Consider any $\left(\mathsf{x},\mathsf{x}_*\right)\in{\widetilde{R}}$, where $\mathsf{x}=\Big(x_1,\ldots,x_{N^{\text{sc}}_{\max}},\\\upsilon_{1},\ldots,\upsilon_{N^{\text{ca}}_{\max}},\widetilde{N}_{1},\ldots,\widetilde{N}_{N_{\max}^{\text{sc}}},\widehat{N}_{1},\ldots,\widehat{N}_{N^{\text{ca}}_{\max}}\Big)$, $\mathsf{x}_*=\Big(x_{*1},\ldots,x_{*N^{\text{sc}}_{\max}},u_{*1},\ldots,u_{*N^{\text{ca}}_{\max}},\widetilde{N}_{*1},\ldots,\widetilde{N}_{*N_{\max}^{\text{sc}}},\widehat{N}_{*1},\\\ldots,\widehat{N}_{*N^{\text{ca}}_{\max}}\Big)$. 
Consider any $\upsilon\in{U(\mathsf{x})}$. Using the relation ${R}$, there exist $u_*\in U_*(\mathsf{x}_*)$ and $\grave{x}_*\in\mathbf{Post}_{u_*}(x_{*})$ such that $\left(\xi_{x\upsilon}(\tau),\grave{x}_*\right)\in{R}$ for any $\left(x,x_*\right)\in{R}$. Now consider any $\mathsf{x}'=\Big(x',x_1,\ldots,x_{N^{\text{sc}}_{\max}-1},\upsilon,\upsilon_{1},\ldots,\upsilon_{N^{\text{ca}}_{\max}-1},\widetilde{N},\widetilde{N}_{1},\\\ldots,\widetilde{N}_{N_{\max}^{\text{sc}}-1},\widehat{N},\widehat{N}_{1}\ldots,\widehat{N}_{N^{\text{ca}}_{\max}-1}\Big)\in\mathbf{Post}_{\upsilon}(\mathsf{x})\subseteq{X}$ for some $\widetilde{N}\in\left[N^{\text{sc}}_{\min};N^{\text{sc}}_{\max}\right]$ and $\widehat{N}\in\left[N^{\text{ca}}_{\min};N^{\text{ca}}_{\max}\right]$ where $x'=\xi_{x\upsilon_k}\left(\tau\right)$ for some given $k\in\left[{N^{\text{ca}}_{\min}};{N^{\text{ca}}_{\max}}\right]$ (cf. Definition $S(\widetilde{\Sigma})$). Because of the relation $R$, there exists $x'_*\in\mathbf{Post}_{u_{*k}}(x_{*1})$ in $S_\params(\Sigma)$ such that $\left(x',x'_*\right)\in{R}$. Hence, due to the definition $S_*(\widetilde\Sigma)$, one can choose $\mathsf{x}'_*=\Big(x'_*,x_{*1},\ldots,x_{*\left(N^{\text{sc}}_{\max}-1\right)},u_*,u_{*1},\ldots,u_{*\left(N^{\text{ca}}_{\max}-1\right)},\widetilde{N},\widetilde{N}_{1},\ldots,\widetilde{N}_{N_{\max}^{\text{sc}}-1},\widehat{N},\widehat{N}_{1},\ldots,\widehat{N}_{N^{\text{ca}}_{\max}-1}\Big)\in\mathbf{Post}_{u_*}(\mathsf{x}_*)\subseteq{X}_*$. Due to the relation ${R}$, one can readily verify that $\mathsf{d}_{Y_\tau}\left( H_\tau(x'),H_\params(x'_*)\right)\leq\varepsilon$. Hence, one gets $\mathsf{d}_{Y}\left(H\left(\mathsf{x}'\right),H_*\left(\mathsf{x}'_*\right)\right)=\mathsf{d}_{Y_\tau}\left(H_\tau\left(x'\right),H_\params\left(x'_*\right)\right)\leq\varepsilon$. Hence, $\left(\mathsf{x}',\mathsf{x}'_*\right)\in{\widetilde{R}}$ implying that condition (iii) in Definition \ref{ASR} holds, which completes the proof.
\end{proof}

\begin{proof}[Proof of Corollary \ref{main_theorem1}]
Using Theorem \ref{main_theorem} one gets that $S_\params(\Sigma)\preceq_{\mathcal{AS}}^{\varepsilon}S_\tau(\Sigma)$ implies \mbox{$S_{*}(\widetilde\Sigma)\preceq_{\mathcal{AS}}^{\varepsilon}S(\widetilde\Sigma)$} equipped with the alternating $\varepsilon$-approximate simulation relation $\widetilde{R}$ as defined in the proof of Theorem \ref{main_theorem}. In a similar way, one can show that $S_\tau(\Sigma)\preceq_{\mathcal{AS}}^{\varepsilon}S_\params(\Sigma)$ implies \mbox{$S(\widetilde\Sigma)\preceq_{\mathcal{AS}}^{\varepsilon}S_{*}(\widetilde\Sigma)$} equipped with the alternating $\varepsilon$-approximate simulation relation $\widetilde{R}^{-1}$ which completes the proof.
\end{proof}

\begin{proof}[Proof of Lemma \ref{less}]
Let $S_*=(X_*,X_{*0},U_*,\rTo_*,Y_*,H_*)$ and $S'_*=(X'_*,X'_{*0},U'_*,{\rTo_*}',Y'_*,H'_*)$. Consider the relation $\widetilde{R}\subseteq X_{*}\times X'_*$ defined by $\left(\mathsf{x}_*,\mathsf{x}'_*\right)\in{\widetilde{R}}$, where $\mathsf{x}_*=\Big(x_{*1},\ldots,x_{*N^{\text{sc}}_{\max}},u_{*1},\ldots,u_{*N^{\text{ca}}_{\max}},\widetilde{N}_{*1},\ldots,\widetilde{N}_{*N_{\max}^{\text{sc}}},\\\widehat{N}_{*1},\ldots,\widehat{N}_{*N^{\text{ca}}_{\max}}\Big)$ and $\mathsf{x}'_*=\Big(x'_{*1},\ldots,x'_{*N^{\text{sc}}_{\max}},u'_{*1},\ldots,u'_{*N^{\text{ca}}_{\max}},N_{\max}^{\text{sc}},\ldots,N_{\max}^{\text{sc}},N^{\text{ca}}_{\max},\ldots,N^{\text{ca}}_{\max}\Big)$, if and only if $\widetilde{N}_{*i}=N_{\max}^{\text{sc}}$, $\forall i\in\left[1;N^{\text{sc}}_{\max}\right]$, $\widehat{N}_{*j}=N_{\max}^{\text{ca}}$, $\forall j\in[1;N^{\text{ca}}_{\max}]$, $x_{*k}=x'_{*k}$, $\forall k\in\left[1;N^{\text{sc}}_{\max}\right]$, and $u_{*k}=u'_{*k}$, $\forall k\in\left[1;N^{\text{ca}}_{\max}\right]$.

Consider $\mathsf{x}_{*0}\Let\big(x_{*0},q,\ldots,q,u_{*0},\ldots,u_{*0},N^{\text{sc}}_{\max},\ldots,N^{\text{sc}}_{\max},N^{\text{ca}}_{\max},\ldots,N^{\text{ca}}_{\max}\big)\in X_{*0}$. One can readily verify that $\mathsf{x}_{*0}\in X'_{*0}$ and, hence, $\left(\mathsf{x}_{*0},\mathsf{x}_{*0}\right)\in \widetilde{R}$. Therefore, condition (i) in Definition \ref{AASR} is satisfied.

Now consider any $\left(\mathsf{x}_*,\mathsf{x}'_*\right)\in{\widetilde{R}}$, where $\mathsf{x}_*=\Big(x_{*1},\ldots,x_{*N^{\text{sc}}_{\max}},u_{*1},\ldots,u_{*N^{\text{ca}}_{\max}},\widetilde{N}_{*1},\ldots,\widetilde{N}_{*N_{\max}^{\text{sc}}},\widehat{N}_{*1},\ldots,\widehat{N}_{*N^{\text{ca}}_{\max}}\Big)$ and $\mathsf{x}'_*=\Big(x'_{*1},\ldots,x'_{*N^{\text{sc}}_{\max}},u'_{*1},\ldots,u'_{*N^{\text{ca}}_{\max}},N_{\max}^{\text{sc}},\ldots,N_{\max}^{\text{sc}},N^{\text{ca}}_{\max},\ldots,N^{\text{ca}}_{\max}\Big)$. Using definitions of $S_*$ and $S'_*$, one obtains $H_*\left(\mathsf{x}_*\right)=H_a(x_{*1})$ and $H'_*\left(\mathsf{x}'_*\right)=H_a(x'_{*1})$. Based on the definition of $\widetilde R$, one obtains $x_{*1}=x'_{*1}$ and, hence, $\mathsf{d}_{Y_a}\left(H_a\left(x_{*1}\right),H_a\left(x'_{*1}\right)\right)=0$. Therefore, condition (ii) in Definition \ref{AASR} is satisfied.

Let us now show that condition (iii) in Definition \ref{AASR} holds. Consider any $\left(\mathsf{x}_*,\mathsf{x}'_*\right)\in{\widetilde{R}}$, where $\mathsf{x}_*=\Big(x_{*1},\ldots,x_{*N^{\text{sc}}_{\max}},u_{*1},\ldots,u_{*N^{\text{ca}}_{\max}},\widetilde{N}_{*1},\ldots,\widetilde{N}_{*N_{\max}^{\text{sc}}},\widehat{N}_{*1},\ldots,\widehat{N}_{*N^{\text{ca}}_{\max}}\Big)$ and $\mathsf{x}'_*=\Big(x'_{*1},\ldots,x'_{*N^{\text{sc}}_{\max}},u'_{*1},\ldots,u'_{*N^{\text{ca}}_{\max}},\\N_{\max}^{\text{sc}},\ldots,N_{\max}^{\text{sc}},N^{\text{ca}}_{\max},\ldots,N^{\text{ca}}_{\max}\Big)$. 
Consider any $u_*\in U_*(\mathsf{x}_*)$. Using the relation ${\widetilde R}$, one can readily verify that $u_*\in{U'_*(\mathsf{x}'_*)}$. Now consider any $\hat{\mathsf{x}}'_*=\Big(\hat x',x'_{*1},\ldots,x'_{*(N^{\text{sc}}_{\max}-1)},u'_*,u'_{*1},\ldots,u'_{*(N^{\text{ca}}_{\max}-1)},N_{\max}^{\text{sc}},\ldots,N_{\max}^{\text{sc}},\\N^{\text{ca}}_{\max},\ldots,N^{\text{ca}}_{\max}\Big)\in\mathbf{Post}_{u_*}(\mathsf{x}'_*)\subseteq{X'_*}$ where $\hat x'\in\mathbf{Post}_{u_{*N^{\text{ca}}_{\max}}}(x'_{*1})$ in $S_a$. Because of the relation $\widetilde R$, one gets $u_{*(N^{\text{ca}}_{\max}-j_*^k)}=u_{*N^{\text{ca}}_{\max}}$ and $x'_{*1}=x_{*1}$ and, hence, $\hat x'\in\mathbf{Post}_{u_{*N^{\text{ca}}_{\max}}}(x_{*1})$ in $S_a$. Hence, one can choose $\hat{\mathsf{x}}_*=\Big(\hat x'_*,x_{*1},\ldots,x_{*\left(N^{\text{sc}}_{\max}-1\right)},u_*,u_{*1},\ldots,u_{*\left(N^{\text{ca}}_{\max}-1\right)},N^{\text{sc}}_{\max},\widetilde{N}_{1},\ldots,\widetilde{N}_{N_{\max}^{\text{sc}}-1},N^{\text{ca}}_{\max},\widehat{N}_{1},\ldots,\widehat{N}_{N^{\text{ca}}_{\max}-1}\Big)$, where $\hat{\mathsf{x}}_*\in\mathbf{Post}_{u_*}(\mathsf{x}_*)\subseteq{X}_*$. Since $\left(\mathsf{x}_*,\mathsf{x}'_*\right)\in{\widetilde{R}}$, one already has $\widetilde{N}_{*i}=N_{\max}^{\text{sc}}$, $\forall i\in\left[1;N^{\text{sc}}_{\max}-1\right]$, $\widehat{N}_{*j}=N_{\max}^{\text{ca}}$, $\forall j\in[1;N^{\text{ca}}_{\max}-1]$, $x_{*k}=x'_{*k}$, $\forall k\in\left[1;N^{\text{sc}}_{\max}-1\right]$, and $u_{*k}=u'_{*k}$, $\forall k\in\left[1;N^{\text{ca}}_{\max}-1\right]$. Therefore, one can readily verify that $\left(\hat{\mathsf{x}}_*,\hat{\mathsf{x}}'_*\right)\in{\widetilde{R}}$ implying that condition (iii) in Definition \ref{AASR} holds which completes the proof.
\end{proof}

\end{document}